\documentclass[preprint,review,10pt]{elsarticle}
\usepackage{amsmath,amssymb,amsfonts,amsthm}
\usepackage{hyperref}
\usepackage{graphicx}
\usepackage{mathrsfs}

\usepackage{multirow}

\usepackage{color}
\usepackage{cases}

\newtheorem{assumption}{Assumption}

\newtheorem{lemma}{Lemma}

\newtheorem{theorem}{Theorem}

\begin{document}

\begin{frontmatter}
\title{
Finite element method for singularly perturbed problems with two parameters on a Bakhvalov-type mesh in 2D
\tnoteref{funding} }

\tnotetext[funding]{
This research is supported by National Science Foundation of China (11771257,11601251).
}

\author[label1] {Jin Zhang\corref{cor1}}
\author[label1] {Yanhui Lv \fnref{cor2}}
\cortext[cor1] {Corresponding author: jinzhangalex@hotmail.com }
\fntext[cor2] {Email: yanhuilv@hotmail.com }
\address[label1]{School of Mathematics and Statistics, Shandong Normal University,
Jinan 250014, China}

\begin{abstract}

For a singularly perturbed elliptic model problem with two small parameters, we analyze  finite element methods of any order on a Bakhvalov-type mesh.    For convergence analysis, we construct a new interpolation by using the characteristics of layers. Besides, a more subtle analysis of the mesh scale near  the exponential layer  is carried out. Based on the interpolation and new analysis of the mesh scale, we prove the optimal convergence order.

\end{abstract}

\begin{keyword}
Singular perturbation\sep Convection--diffusion equation \sep Finite element method \sep Bakhvalov-type mesh \sep Two parameters
\end{keyword}

\end{frontmatter}
\section{Introduction}
In this paper, we reconsider the singularly perturbed elliptic problem in \cite{Tefa1Roos2:2008-elliptic}
\begin{align}
&Lu:=-\varepsilon_1\Delta u+\varepsilon_2b(x)u_x+c(x)u=f(x,y)\quad \text{in $\Omega:=(0,1)\times(0,1)$},\label{eq:I-condition-1}\\
&u\vert_{\partial\Omega}=0,\nonumber
\end{align}
with
\begin{align}
b(x)\ge \lambda>0,\quad c(x)\ge \beta>0\quad \text{for $x\in [0,1]$},\label{eq:I-condition-2}\\
c(x)-\frac{1}{2}\varepsilon_2b'(x)\ge \gamma>0,\label{eq:I-condition-3}\\
f(0,0)=f(0,1)=f(1,1)=f(1,0)=0\label{eq:I-condition-4},
\end{align}
where $b$, $c$, and $f$ are sufficiently smooth functions and $\lambda$, $\beta$, $\gamma$ are constants.  
Here we only discuss the case of $0<\varepsilon_1,\varepsilon_2\ll 1$ (see \cite{OMalley:1967-Two-parameter}), and if you are interested in the case of $\varepsilon_2=0$ and $\varepsilon_2=1$, you can refer to \cite{Tefa1Roos2:2008-elliptic} and its references. For the sake of analysis, we use $b(x)$ and $c(x)$ in our problem instead of $b(x,y)$ and $c(x,y)$, because that doesn't affect the properties of the true solution(see\cite{Teo1Brd2Rra3Zar4:2018-SDFEM}), and this is also the case in the equations of \cite{Tefa1Roos2:2007-elliptic}, \cite{Tefa1Roos2:2008-elliptic}, and \cite{Brd1Zar2Teo3:2016-singular}.


 Moreover, the conditions \eqref{eq:I-condition-2} and \eqref{eq:I-condition-3} ensure that  there exists a unique solution $u\in C^{3,\alpha_0}(\Omega)$ with $\alpha_0\in(0,1)$,  which is characterized by exponential layers at $x = 0$ and $x = 1$,  parabolic layers at  $y=0$ and $y=1$, and corner layers at four corners of the domain. 
For the treatment of boundary layer,  researchers usually use a class of special meshes which are very fine on the layer region of the solution. Compared to the quasi-uniform mesh, this kind of mesh can capture the change of layers better.  Among those the most representative ones are Shishkin mesh \cite{shishkin1990grid} and Bakhvalov mesh \cite{Bakhvalov:1969-Towards}, and experiments show that the convergence order of numerical solutions on Bakhvalov mesh is better.

In fact, up to now there are few articles about  finite element method on Bakhvalov mesh, because directly applying the Lagrangian interpolation commonly used in numerical analysis to the Bakhvalov mesh is not workable.  
In \cite{Roo1Han2:2006-Error}, Roos clearly stated the difficulty of convergence analysis on Bakhvalov-type mesh in 1D, and obtained the optimal convergence order  by using quasi-interpolation.  Brdar and Zarin analyzed a singularly perturbed   problem with two-parameter in 1D  using the same method in \cite{Brda1Zari2:2016-singularly}.  However, Roos'  method is powerless in the face of higher-order finite element methods or higher-dimensional problems. Recently,  Zhang and Liu  proposed a new interpolation which is much simpler to construct and analyze for the one-dimensional one-parameter problem in \cite{Zha1Liu2:2020-Opt}, and can be directly extended to higher-order cases. But, when the idea is applied to two-dimension, the boundary conditions need to be properly corrected to meet the homogeneous Dirichlet boundary conditions. Standard analysis is not successful for these corrections (see \cite{Zhan1Liu2:2020-Convergence-arXiv-2}).
%

In this paper, we define a new interpolation according to the characteristics of the layers for convergence analysis of optimal order. In view of the difficulties brought   by the corrections on the boundary,  we make use of a new estimation of the mesh scale near the Bakhvalov-type transition point and take the structure of the mesh near the boundary into account. Furthermore, we use different techniques to handle error estimations in different subdomains and then obtain the optimal convergence order on the Bakhvalov-type mesh. 

The rest of this article is organized as follows. In the section 2, the prior estimation of the solution of the continuous problem is given. In the section 3, we will construct the Bakhvalov-type mesh, and give some mesh properties, and finally establish the finite element method. The new interpolation will appear in the section 4, and we will also prove some results of Lagrangian interpolation error. The convergence analysis is carried out in section 5. Finally, numerical experiments are given in section 6 to verify our conclusion.

Throughout the paper, we shall use $C$ to denote a generic positive constant independent of  $\varepsilon_1$, $\varepsilon_2$ and  $N$, which can take different values at different places. For any domain $D$ of $\Omega$,  we use the standard notation for Banach spaces $L^p(D)$, Sobolev spaces $W^{k, p}(D)$, $H^k(D)=W^{k, 2}(D)$. Define $\Vert \cdot\Vert_{\infty,D}$ to be  $\Vert\cdot\Vert_{L^{\infty}(D)}$, $\Vert \cdot\Vert_D$ to be $\Vert \cdot\Vert_{L^2(D)}$, and $\vert\cdot\vert_{D}$ to be the seminorms of $\Vert\cdot\Vert_{H^1(D)}$; The scalar product in $L^2(D)$ is denoted with $(\cdot, \cdot)_D$. And we will drop the subscript $D$ from the notation for simplicity when $D=\Omega.$

 
\section{A priori estimates of solution of the continuous problem}
Compared with the parabolic layer, the exponential layer changes more dramatically. So in order to describe the exponential layers at $x=0$ and $x=1$, we introduce the characteristic equation as following 
$$-\varepsilon_1g^2(x)+\varepsilon_2b(x)g(x)+c(x)=0.$$
This equation defines two continuous functions $g_0, g_1:[0, 1]\rightarrow \mathbb{R}$ with $g_0(x)<0$, $g_1(x)>0$. Let
$$\mu_0=-\max\limits_{0\leq x\leq1} g_0(x), \quad \mu_1=\min\limits_{0\leq x\leq1} g_1(x).$$
For the sake of simplicity, we take 
$$\mu_0=\frac{-\varepsilon_2b_*+\sqrt{\varepsilon_2^2b_*^2+4\varepsilon_1\beta}}{2\varepsilon_1}, \quad 
\mu_1=\frac{\varepsilon_2\lambda+\sqrt{\varepsilon_2^2\lambda^2+4\varepsilon_1\beta}}{2\varepsilon_1},$$ 
with $b_*=\max\limits_{0\leq x\leq 1}b(x)$,  which is the same as \cite[(6)]{Tefa1Roos2:2008-elliptic}.
Then we give some properties of $\mu_0$  and $\mu_1$ (see \cite{Tefa1Roos2:2007-elliptic}): 
\begin{align}
&\mu_0\leq\mu_1,\qquad \qquad \max\{\mu_0^{-1}, \varepsilon_1\mu_1\}\leq C(\varepsilon_2+\varepsilon_1^{\frac{1}{2}}),\label{cos:conclusions-1}\\
&\varepsilon_2\mu_0\leq \lambda^{-1}\Vert c\Vert_{\infty},\qquad\qquad\varepsilon_2(\varepsilon_1\mu_1)^{-\frac{1}{2}}
\leq C\varepsilon_2^{\frac{1}{2}}\label{cos:conclusions-2}.
\end{align}
These properties will play an important role in the subsequent analysis.

In this paper we assume that
\begin{equation}\label{eq:mu01-assume}
\mu_1^{-1}\leq\mu_0^{-1}\leq N^{-1},
\end{equation}
and it is worth noting that there is no such limitation in practice.
By direct computations of \eqref{eq:mu01-assume}, we can obtain
\begin{equation}\label{eq:varepsilon12-condition}
\varepsilon_1\leq c_0N^{-2},\qquad \varepsilon_2\leq c_1N^{-1},
\end{equation}
with  $c_0=\beta$ and $c_1=\frac{\beta}{b_*}$.

On the basis of the prior estimation of the solution  given in \cite{Tefa1Roos2:2007-elliptic}, we make the following assumption about the decomposition of the solution and the prior estimation of each component. In the subsequent analysis, $k$ is a fixed positive integer and $k\geq1$.
%

\begin{assumption}\label{eq:priori estimates of u}
Let there be given elliptic problem \eqref{eq:I-condition-1} on the unit square $\bar\Omega$ satisfying conditions \eqref{eq:I-condition-2}-\eqref{eq:I-condition-4}, and let $p\in(0,1)$ and $k_0\in(0,\frac{1}{2})$ be arbitrary. Assume that
\begin{equation*}\label{eq:bk-condition}
2\Vert b'\Vert_{{\infty}}\varepsilon_2\leq k_0(1-p)\beta.
\end{equation*}
Furthermore, let $\delta$ be a positive constant satisfying
\begin{equation*}\label{eq:delta-condition}
\delta^2\leq \frac{(1-p)\beta}{2}.
\end{equation*}
Then the solution $u$ of  problem \eqref{eq:I-condition-1} can be decomposed as 
\begin{equation*}\label{eq:decompose of u}
u=S+E_{10}+E_{11}+E_{20}+E_{21}+E_{31}+E_{32}+E_{33}+E_{34},
\end{equation*}
where for all $(x,y)\in \bar\Omega$ and $0\leq i+j\leq k+1,$ the regular part $S$ satisfies
\begin{equation*}\label{eq:S-property}
\left\vert \frac{\partial^{i+j}S}{\partial x^i\partial y^j}\right\vert\leq C,
\end{equation*}
the exponential and parabolic layer components satisfy
\begin{align}
&\left\vert \frac{\partial^{i+j}E_{10}}{\partial x^i\partial y^j}\right\vert\leq C\mu_0^ie^{-p\mu_0x},\label{eq:E10-property}\\
&\left\vert \frac{\partial^{i+j}E_{11}}{\partial x^i\partial y^j}\right\vert\leq C\mu_1^ie^{-p\mu_1(1-x)},\label{eq:E11-property}\\
&\left\vert \frac{\partial^{i+j}E_{20}}{\partial x^i\partial y^j}\right\vert\leq C\varepsilon_1^{-\frac{j}{2}}e^{-\frac{\delta y}{\sqrt{\varepsilon_1}}},\nonumber\\
&\left\vert \frac{\partial^{i+j}E_{21}}{\partial x^i\partial y^j}\right\vert\leq C\varepsilon_1^{-\frac{j}{2}}e^{-\frac{\delta (1-y)}{\sqrt{\varepsilon_1}}},\nonumber
\end{align}
while the corner layer components satisfy the following estimates
\begin{align}
&\left\vert \frac{\partial^{i+j}E_{31}}{\partial x^i\partial y^j}\right\vert\leq C\varepsilon_1^{-\frac{j}{2}}\mu_0^ie^{-p\mu_0x}e^{-\frac{\delta y}{\sqrt{\varepsilon_1}}},\label{eq:E31-property}\\
&\left\vert \frac{\partial^{i+j}E_{32}}{\partial x^i\partial y^j}\right\vert\leq C\varepsilon_1^{-\frac{j}{2}}\mu_1^ie^{-p\mu_1(1-x)}e^{-\frac{\delta y}{\sqrt{\varepsilon_1}}},\label{eq:E32-property}\\
&\left\vert \frac{\partial^{i+j}E_{33}}{\partial x^i\partial y^j}\right\vert\leq C\varepsilon_1^{-\frac{j}{2}}\mu_1^ie^{-p\mu_1(1-x)}e^{-\frac{\delta (1-y)}{\sqrt{\varepsilon_1}}},\nonumber\\
&\left\vert \frac{\partial^{i+j}E_{34}}{\partial x^i\partial y^j}\right\vert\leq C\varepsilon_1^{-\frac{j}{2}}\mu_0^ie^{-p\mu_0x}e^{-\frac{\delta (1-y)}{\sqrt{\varepsilon_1}}}.\nonumber
\end{align}
\end{assumption}
\section{Bakhvalov-type mesh and finite element method}

\subsection{Bakhvalov-type mesh}
Let $N\in \mathbb{N}, N\ge 8$, can be divisible 4. Define
\begin{align*}
\sigma_{x,i}:=\frac{\tau}{p\mu_j}\ln \mu_j\; i=0, 1,\quad \text{and}\quad
\sigma_y:=\frac{\tau}{\delta}\sqrt{\varepsilon_1}\ln \frac{1}{\sqrt{\varepsilon_1}},
\end{align*}
where $\tau\geq k+1$ is a user-chosen parameter and $p\in(0,1)$ is the parameter from Assumption \ref{eq:priori estimates of u}. On $x$--axis, we set $\sigma_{x,0}$ and $1-\sigma_{x,1}$ as  transition points, where  the mesh changes from fine to coarse and viceversa. On $y$-- axis, we set $\sigma_y$ and $1-\sigma_y$ as transition points.
For technical reasons, we also assume 
\begin{align}\label{eq:transition-points}
\sigma_{x,i}\leq\frac{1}{4}\ \  i=0, 1,
\quad\text{and}\quad\sigma_y\leq\frac{1}{4}.
\end{align}
%

Now we define a Bakhvalov-type mesh for problem \eqref{eq:I-condition-1}, which is introduced in \cite{Roo1Han2:2006-Error}. The mesh points $x_i$, $i=0, 1, \cdots, N$, are defined by    
\begin{equation}\label{eq:mesh-x-points}
x_i=\begin{cases}
 \frac{\tau}{p\mu_0}\varphi_0(t_i)\qquad\qquad i=0, 1, \cdots, \frac{N}{4},\\
 \sigma_{x,0}+2(t_i-\frac{1}{4})(1-\sigma_{x,0}-\sigma_{x,1})\qquad i=\frac{N}{4},\frac{N}{4}+1,\cdots,\frac{3N}{4},\\
 1-\frac{\tau}{p\mu_1}\varphi_1(t_i)\qquad\qquad i=\frac{3N}{4}, \frac{3N}{4}+1, \cdots, N,
\end{cases}
\end{equation}
where $t_i=\frac{i}{N}$, $i=0, 1, \cdots, N$ and
$$
\varphi_0(t)=-\ln(1-4(1-\mu_0^{-1})t),\quad
\varphi_1(t)=-\ln(1-4(1-\mu_1^{-1})(1-t)).
$$
It can be seen from \eqref{eq:mesh-x-points} that the mesh is graded on $[x_0,\sigma_{x,0}]$ and $[1-\sigma_{x,1},x_N]$, and the mesh is uniform on $[\sigma_{x,0},1-\sigma_{x,1}]$.
The mesh points $y_j$, $j=0, 1, \cdots, N$, are defined by
\begin{equation}\label{eq:mesh-y-points}
y_j=\begin{cases}
 \frac{\tau}{\delta}\sqrt{\varepsilon_1}\phi_0(t_j)\qquad\qquad j=0, 1, \cdots, \frac{N}{4},\\
 \sigma_y+2(t_j-\frac{1}{4})(1-2\sigma_y)\qquad j=\frac{N}{4},\frac{N}{4}+1,\cdots,\frac{3N}{4},\\
 1-\frac{\tau}{\delta}\sqrt{\varepsilon_1}\phi_1(t_j)\qquad\qquad j=\frac{3N}{4}, \frac{3N}{4}+1, \cdots, N,
\end{cases}
\end{equation}
where $t_j=\frac{j}{N}$, $j=0, 1, \cdots, N$ and
$$
\phi_0(t)=-\ln(1-4(1-\sqrt{\varepsilon_1})t),\quad
\phi_1(t)=-\ln(1-4(1-\sqrt{\varepsilon_1})(1-t)).
$$
Similar to the $x$-direction mesh layout, from \eqref{eq:mesh-y-points} we could see that the mesh is graded on $[y_0,\sigma_y]$ and $[1-\sigma_y,y_N]$, and the mesh is uniform on $[\sigma_y,1-\sigma_y].$ With  mesh points $\{(x_i,y_j)\}$,  we obtain a tensor-product rectangular mesh $\mathcal{T}$.


Set $h_{x,i}:=x_{i+1}-x_{i}$ and $h_{y,j}:=y_{j+1}-y_{j}$  are the mesh sizes in the $x$ and $y$ directions,  respectively. Also set $\mathscr{T}_{i,j}=[x_i,x_{i+1}]\times[y_j,y_{j+1}]$ be any mesh rectangle in $\mathcal{T}$.

According to  \cite[Lemmas 2 and 3]{Zha1Liu2:2020-Opt}, we have the following  three  lemmas.
\begin{lemma}\label{eq:mesh-x-sizes}
From \eqref{eq:mesh-x-points}, we can get the mesh size in the $x$ direction as follows
\begin{align*}
&C\mu_0^{-1}N^{-1}\leq h_{x,0}\leq h_{x,1}\leq \cdots\leq h_{x,\frac{N}{4}-2},\\
&\frac{\tau}{4p}\mu_0^{-1}\leq h_{x,\frac{N}{4}-2}\leq\frac{\tau}{p}\mu_0^{-1},\\
&\frac{\tau}{2p}\mu_0^{-1}\leq h_{x,\frac{N}{4}-1}\leq \frac{4\tau}{p}N^{-1},\\
&N^{-1}\leq h_{x,i}\leq 2N^{-1}\qquad \frac{N}{4}\leq i\leq \frac{3N}{4}-1,\\
&\frac{\tau}{2p}\mu_1^{-1}\leq h_{x,\frac{3N}{4}}\leq \frac{4\tau}{p}N^{-1},\\
&\frac{\tau}{4p}\mu_1^{-1}\leq h_{x,\frac{3N}{4}+1}\leq\frac{\tau}{p}\mu_1^{-1},\\
&h_{x,\frac{3N}{4}+1}\ge h_{x,\frac{3N}{4}+2}\ge \cdots\ge h_{x,N-1}\ge C\mu_1^{-1}N^{-1},\\
&1-x_{\frac{3N}{4}+2}\leq C\mu_1^{-1}\ln N.
\end{align*}
\end{lemma}
\begin{lemma}\label{eq:mesh-y-sizes}
From \eqref{eq:mesh-y-points}, we can get the mesh size in the $y$ direction as follows
\begin{align*}
&C\sqrt{\varepsilon_1}N^{-1}\leq h_{y,0}\leq h_{y,1}\leq \cdots\leq h_{y,\frac{N}{4}-2},\\
&C_1\sqrt{\varepsilon_1}\leq h_{y,\frac{N}{4}-2}\leq\frac{\tau}{p}\sqrt{\varepsilon_1},\\
&C_2\sqrt{\varepsilon_1}\leq h_{y,\frac{N}{4}-1}\leq \frac{4\tau}{p}N^{-1},\\
&N^{-1}\leq h_{y,j}\leq 2N^{-1}\qquad \frac{N}{4}\leq j\leq \frac{3N}{4}-1,\\
&C_2\sqrt{\varepsilon_1}\leq h_{y,\frac{3N}{4}}\leq \frac{4\tau}{p}N^{-1},\\
&C_1\sqrt{\varepsilon_1}\leq h_{y,\frac{3N}{4}+1}\leq\frac{\tau}{p}\sqrt{\varepsilon_1},\\
&h_{y,\frac{3N}{4}+1}\ge h_{y,\frac{3N}{4}+2}\ge \cdots\ge h_{y,N-1}\ge C\sqrt{\varepsilon_1}N^{-1},
\end{align*}
where $C_1=\frac{\tau}{\delta(\sqrt{c_0}+8)}$, $C_2=\frac{\tau}{\delta(\sqrt{c_0}+4)}.$
\end{lemma}
\begin{lemma}\label{lem:Error-condition}
For $0\leq i\leq\frac{N}{4}-2$ and $0\leq m\leq\tau$, one has
\begin{equation}\label{lem:e-E10}
h_{x,i}^me^{-p\mu_0x_i}\leq C\mu_0^{-m}N^{-m}.
\end{equation}
For $\frac{3N}{4}+1\leq i\leq N-1$ and $0\leq m\leq\tau$, one has
\begin{equation}\label{lem:e-E11}
h_{x,i}^me^{-p\mu_1(1-x_{i+1})}\leq C\mu_1^{-m}N^{-m}.
\end{equation}
For $0\leq j\leq\frac{N}{4}-2$ and $0\leq m\leq\tau$, one has
\begin{equation}\label{lem:e-E20}
h_{y,j}^me^{-\frac{\delta y_j}{\sqrt{\varepsilon_1}}}\leq C\varepsilon_1^{\frac{m}{2}}N^{-m}.
\end{equation}
For $\frac{3N}{4}+1\leq j\leq N-1$ and $0\leq m\leq\tau$, one has
\begin{equation*}\label{lem:e-E21}
h_{y,j}^me^{-\frac{\delta (1-y_{j+1})}{\sqrt{\varepsilon_1}}}\leq C\varepsilon_1^{\frac{m}{2}}N^{-m}.
\end{equation*}
\end{lemma}

Also we need to re-estimate  $h_{x,\frac{3N}{4}}$ for our convergence analysis.
\begin{lemma}\label{eq:better-mesh}
For any fixed $\eta\in(0,1]$, one has
\begin{equation*}
h_{x,\frac{3N}{4}}\leq C\mu_1^{\eta-1}N^{-\eta}.
\end{equation*}
\end{lemma}
\begin{proof}
For any fixed $\eta\in(0,1]$, standard arguments show
\begin{equation*}\label{eq:alpha}
\ln x\leq \frac{x^{\eta}}{\eta}\quad x\in[1,+\infty).
\end{equation*}
Combine  \eqref{eq:mesh-x-points} to get
\begin{equation*}
\begin{aligned}
h_{x,\frac{3N}{4}}&= \frac{\tau}{p\mu_1}\ln \frac{\mu_1^{-1}+4(1-\mu_1^{-1})N^{-1}}{\mu_1^{-1}}\\
&\leq \frac{\tau}{p\mu_1}\ln ( {N^{-1}}{\mu_1})\leq \frac{1}{\eta}\frac{\tau}{p\mu_1}{N^{-\eta}}{\mu_1^{\eta}}\\
&\leq C\mu_1^{\eta-1}N^{-\eta}.
\end{aligned}
\end{equation*}
\end{proof}

\subsection{Finite element method}
The weak form of problem \eqref{eq:I-condition-1} is to find $u\in H^1_0(\Omega)$ such that
\begin{equation}\label{eq:weak form-1d}
a(u, v)=(f, v) \quad \forall v\in H^1_0(\Omega),
\end{equation}
where
\begin{equation}\label{eq:bilinear form}
a(u,v):=\varepsilon_1 (\nabla u, \nabla v)+(\varepsilon_2bu_x+cu, v) 
\end{equation}
and $(\cdot,\cdot)$ denotes the standard scalar product in $L^2(\Omega)$.

Define the finite element space on the Bakhvalov-type  mesh
\begin{equation*}\label{eq:VN}
V^{N}=\{w\in C(\bar{\Omega}): w\vert_{\partial\Omega}=0,
 w\vert_{\mathscr{T}}\in \mathcal{Q}_k(\mathscr{T}) \quad \forall
\mathscr{T}\in\mathcal{T} \} ,
\end{equation*}
where $\mathcal{Q}_k(\mathscr{T})=\sum\limits_{0\leq i,j\leq k}\alpha_{ij}x^iy^j$ with constants $\alpha_{ij}\in\mathbb{R}$.

The finite element method for \eqref{eq:weak form-1d} is to find $u^N\in V^{N}$ such that
\begin{equation}\label{eq:FE-1d}
a(u^N,v^N)=(f,v^N) \quad \forall v^N\in V^N.
\end{equation}
The energy norm associated with $a(\cdot,\cdot)$ is defined by
\begin{equation*}\label{eq:energy norm}
\Vert v \Vert_{ E}^2:=
\varepsilon_1 \vert  v \vert^2_1+ \Vert v \Vert^2  \quad \forall v\in H^1(\Omega).
\end{equation*}
Using \eqref{eq:I-condition-3},  it's easy to prove coercivity
\begin{equation}\label{eq:coercivity}
a(v^N,v^N) \ge C \Vert v^N \Vert_{E}^2\quad \text{for all $v^N\in V^N$}.
\end{equation}
It follows that $u^N$ is well defined by \eqref{eq:FE-1d} (see \cite{Bren1Scot2:2008-mathematical} and references therein). 
\section{Interpolation errors}
In this section we will introduce a new interpolation. The structure of this interpolation is similar to one in \cite{Zha1Liu2:2020-Opt}.
Set $x_i^s:=x_i+\frac{s}{k}h_{x,i}$ and $y_j^t:=y_j+\frac{t}{k}h_{y,j}$ for $i,j=0,1,\cdots,N-1$ and $s,t=1,2,\cdots,k.$

For any $v\in C^0(\bar{\Omega})$ its Lagrange interpolation $v^I\in V^N$ on the  Bakhvalov-type  mesh is defined by
$$\begin{aligned}
v^I(x,y)=&\sum_{i=0}^{N-1}\sum_{s=0}^{k-1}\left(
\sum_{j=0}^{N-1}\sum_{t=0}^{k-1}v(x_i^s,y_j^t)\theta_{i,j}^{s,t}(x,y)+v(x_i^s,y_N^0)\theta_{i,N}^{s,0}(x,y)\right)\\
&+\sum_{j=0}^{N-1}\sum_{t=0}^{k-1}v(x_N^0,y_j^t)\theta_{N,j}^{0,t}(x,y)+v(x_N^0,y_N^0)\theta_{N,N}^{0,0}(x,y),
\end{aligned}$$
where $\theta_{i,j}^{s,t}(x,y)\in V^N$ is the piecewise $k$th-order Lagrange basis function satisfying
the well-known delta properties associated with the nodes $(x_i^s,y_j^t)$.
We define the interpolation $\Pi u$ to the solution $u$ by
\begin{equation}\label{eq:Interpolation-u}
\Pi u:=S^I+E_{10}^I+\pi_{11} E_{11}+E_{20}^I+E_{21}^I+E_{31}^I+\pi_{32}E_{32}+\pi_{33}E_{33}+E_{34}^I,
\end{equation}
where
\begin{equation}\label{eq:define-piE}
\pi_iE_i(x,y)=E_i^I-PE_i+\Theta E_i\quad\text{for $i=11,32,33$}
\end{equation}
with 
\begin{align*}
&(PE_{a})(x,y)=\sum_{i=\frac{3N}{4}}\sum_{s=1}^{k}\left(\sum_{j=0}^{N-1}\sum_{t=0}^{k-1}E_{a}(x_i^s,y_j^t)\theta_{i,j}^{s,t}(x,y)+E_{a}(x_i^s,y_N^0)\theta_{i,N}^{s,0}(x,y)\right)\\
&(\Theta E_{a})(x,y)=\sum_{s=1}^{k}E_{a}(x_{\frac{3N}{4}}^s,y_0^0)\theta_{\frac{3N}{4},0}^{s,0}+\sum_{s=1}^{k}E_{a}(x_{\frac{3N}{4}}^s,y_N^0)\theta_{\frac{3N}{4},N}^{s,0}\quad a=11,32,33.
\end{align*}

From \eqref{eq:Interpolation-u} and \eqref{eq:define-piE} we can easily get $\Pi u\in V^N$ and 
\begin{equation}\label{eq:piu}
\Pi u=u^I-\sum\limits_{i=11,32,33}(PE_i-\Theta E_i).
\end{equation}

Next, we will prove the Lagrange interpolation estimation.
From \cite[Theorem 2.7]{Ape1Tho2:1999-Ani}, we have the following anisotropic interpolation results.
\begin{lemma}\label{eq:Interpolation-error}
Let $\mathscr{T}\in\mathcal{T}$ and $v\in H^{k+1}(\mathscr{T})$. Then there exists a constant C such that Lagrange interpolation $v^I$ satisfies
\begin{align*}
\Vert v-v^I\Vert_{\mathscr{T}}\leq C\sum_{i+j=k+1}h_{x,\mathscr{T}}^ih_{y,\mathscr{T}}^j\left\| \frac{\partial^{k+1}v}{\partial x^i\partial y^j}\right\|_{\mathscr{T}},\\
\Vert (v-v^I)_x\Vert_{\mathscr{T}}\leq C\sum_{i+j=k}h_{x,\mathscr{T}}^ih_{y,\mathscr{T}}^j\left\| \frac{\partial^{k+1}v}{\partial x^{i+1}\partial y^j}\right\|_{\mathscr{T}},\\
\Vert (v-v^I)_y\Vert_{\mathscr{T}}\leq C\sum_{i+j=k}h_{x,\mathscr{T}}^ih_{y,\mathscr{T}}^j\left\| \frac{\partial^{k+1}v}{\partial x^i\partial y^{j+1}}\right\|_{\mathscr{T}},
\end{align*}
where $h_{x,\mathscr{T}}$ and $h_{y,\mathscr{T}}$ are respectively the mesh size in $x$ direction and $y$ direction on the  rectangular interval $\mathscr{T}$.
\end{lemma}

\begin{lemma}\label{eq:Interpolation-L2}
Assume $\tau\geq k+1$. On Bakhvalov-type mesh $\mathcal{T}$, one has
\begin{align*}
&\Vert E_i-E_i^I\Vert\leq CN^{-(k+1)}\qquad i=10,11,20,21,31,32,33,34.
\end{align*}
\end{lemma}
\begin{proof}
To consider $\Vert E_{10}-E_{10}^I\Vert$, we decompose it as follows
$$
\begin{aligned}
\Vert E_{10}-E_{10}^I\Vert^2&=\Vert E_{10}-E_{10}^I\Vert^2_{[x_0,x_{\frac{N}{4}-1}]\times[0,1]}
 +\Vert E_{10}-E_{10}^I\Vert^2_{[x_{\frac{N}{4}-1},x_N]\times[0,1]}\\
&=:A_1+A_2.
\end{aligned}
$$
Using  \eqref{eq:E10-property}, Lemmas \ref{eq:mesh-x-sizes}, \ref{eq:mesh-y-sizes}, \ref{eq:Interpolation-error} and \eqref{lem:e-E10} with $m=l$ we obtain
\begin{equation}\label{eq:A1}
\begin{aligned}
A_1=&\sum_{i=0}^{\frac{N}{4}-2}\sum_{j=0}^{N-1}\Vert E_{10}-E_{10}^I\Vert_{\mathscr{T}_{i,j}}^2\leq C\sum_{i=0}^{\frac{N}{4}-2}\sum_{j=0}^{N-1}\sum\limits_{l+r=k+1}h_{x,i}^{2l}h_{y,j}^{2r}\left\|\frac{\partial^{k+1}E_{10}}{\partial x^l\partial y^r}\right\|_{\mathscr{T}_{i,j}}^2\\
&\leq C\sum_{i=0}^{\frac{N}{4}-2}\sum_{j=0}^{N-1}\sum\limits_{l+r=k+1}h_{x,i}^{2l}h_{y,j}^{2r}(\mu_0^{2l}e^{-2p\mu_0x_i}h_{x,i}h_{y,j})\\
&\leq C\sum_{i=0}^{\frac{N}{4}-2}\sum_{j=0}^{N-1}\sum\limits_{l+r=k+1}(\mu_0^{-2l}N^{-2l})\mu_0^{2l}h_{x,i}h_{y,j}^{2r+1}\\
&\leq C\sum_{i=0}^{\frac{N}{4}-2}\sum_{j=0}^{N-1}\mu_0^{-1}N^{-2(k+1)-1}\\
&\leq C\mu_0^{-1}N^{-(2k+1)},
\end{aligned}
\end{equation}
and after a simple calculation, we get $\vert E_{10}(x,y)\vert_{[x_{\frac{N}{4}-1},x_N]\times[0,1]}\leq CN^{-\tau}.$ 
Then  the triangle inequality  yields
\begin{equation}\label{eq:A2}
\begin{aligned}
A_2&\leq C(\Vert E_{10}\Vert_{[x_{\frac{N}{4}-1},x_N]\times[0,1]}^2+\Vert E_{10}^I\Vert_{[x_{\frac{N}{4}-1},x_N]\times[0,1]}^2)\\
&\leq C(\Vert E_{10}\Vert_{\infty,[x_{\frac{N}{4}-1},x_N]\times[0,1]}^2+\Vert E_{10}^I\Vert_{\infty,[x_{\frac{N}{4}-1},x_N]\times[0,1]}^2)\\
&\leq C\Vert E_{10}\Vert_{\infty,[x_{\frac{N}{4}-1},x_N]\times[0,1]}^2\\
&\leq CN^{-2\tau}.
\end{aligned}
\end{equation}
From \eqref{eq:mu01-assume},\eqref{eq:A1} and \eqref{eq:A2} we could prove our conclusion. Using the same  method we could get the  estimates of $\Vert E_i-E_i^I\Vert$ with $i=11,20,21,31,32,33,34$. For the cases of $i=31,32,33,34$, we divide the whole interval into three pieces not two pieces in the case of $E_{10}-E_{10}^I$. For example,
for $\Vert E_{31}-E_{31}^I\Vert$, we can break it down into  
\begin{equation*}\label{eq:E31-decompose}
\begin{aligned}
\Vert E_{31}-E_{31}^I\Vert^2&=\Vert E_{31}-E_{31}^I\Vert_{[x_0,x_{\frac{N}{4}-1}]\times[y_0,y_{\frac{N}{4}-1}]}^2\\
&+\Vert E_{31}-E_{31}^I\Vert^2_{[x_0,x_{\frac{N}{4}-1}]\times[y_{\frac{N}{4}-1},y_N]}+\Vert E_{31}-E_{31}^I\Vert^2_{[x_{\frac{N}{4}-1},x_N]\times[0,1]}\\
&=:B_1+B_2+B_3.
\end{aligned}
\end{equation*}
Similar to \eqref{eq:A1}, we get
\begin{equation*}\label{eq:B1}
\begin{aligned}
B_1\leq C\varepsilon_1^{\frac{1}{2}}\mu_0^{-1}N^{-2k}.
\end{aligned}
\end{equation*}
And similar to \eqref{eq:A2}, one has
\begin{equation*}\label{eq:B2}
B_2+B_3\leq CN^{-2\tau}.
\end{equation*}

\end{proof}

\begin{lemma}\label{eq:Interpolation-E-energy}
Assume $\tau\geq k+1$. On Bakhvalov-type mesh $\mathcal{T}$, one has
\begin{align*}
&\Vert E_i-E_i^I\Vert_E\leq C(\varepsilon_1^{\frac{1}{2}}+\varepsilon_2)^{\frac{1}{2}}N^{-k}+CN^{-(k+1)}\quad i=10,11,20,21,\\
&\Vert PE_{11}\Vert_E\leq CN^{-\tau-\frac{1}{2}},\\
&\Vert\Theta E_{11}\Vert_E\leq C\varepsilon_1^{\frac{1}{4}}N^{-\tau}.
\end{align*}
\end{lemma}
\begin{proof}
We only consider $\Vert E_{10}-E_{10}^I\Vert_E$, because the remaining terms could be analyzed in a similar way. Clearly, one has
$$\vert E_{10}-E_{10}^I\vert_1^2=\Vert (E_{10}-E_{10}^I)_x\Vert^2+\Vert (E_{10}-E_{10}^I)_y\Vert^2.$$ 

From \eqref{eq:E10-property},   Lemmas \ref{eq:mesh-y-sizes},  \ref{eq:Interpolation-error} and \eqref{lem:e-E10} with $m=l+\frac{1}{2}$ we could obtain
\begin{equation}\label{eq:E10x-1}
\begin{aligned}
&\Vert (E_{10}-E_{10}^I)_x\Vert_{[x_0,x_{\frac{N}{4}-1}]\times[0,1]}^2
=\sum_{i=0}^{\frac{N}{4}-2}\sum_{j=0}^{N-1}\Vert(E_{10}-E_{10}^I)_x\Vert_{\mathscr{T}_{i,j}}^2\\
&\leq C\sum_{i=0}^{\frac{N}{4}-2}\sum_{j=0}^{N-1}\sum\limits_{l+r=k}h_{x,i}^{2l}h_{y,j}^{2r}\left\Vert\frac{\partial^{k+1}E_{10}}{\partial x^{l+1}\partial y^r}\right\Vert_{\mathscr{T}_{i,j}}^2\\
&\leq C\sum_{i=0}^{\frac{N}{4}-2}\sum_{j=0}^{N-1}\sum\limits_{l+r=k}h_{x,i}^{2l}h_{y,j}^{2r}(\mu_0^{2(l+1)}e^{-2p\mu_0x_i}h_{x,i}h_{y,j})\\
&\leq C\sum_{i=0}^{\frac{N}{4}-2}\sum_{j=0}^{N-1}\sum\limits_{l+r=k}(\mu_0^{-2l-1}N^{-2l-1})\mu_0^{2(l+1)}h_{y,j}^{2r+1}\\
&\leq C\sum_{i=0}^{\frac{N}{4}-2}\sum_{j=0}^{N-1}\mu_0N^{-2(k+1)}\\
&\leq C\mu_0N^{-2k}.
\end{aligned}
\end{equation}
Note $\Vert (E_{10})_x\Vert_{[x_{\frac{N}{4}-1},x_N]\times[0,1]}\leq C\mu_0^{\frac{1}{2}}N^{-\tau}.$ Then one has 
\begin{equation}\label{eq:E10x-2}
\begin{aligned}
&\Vert (E_{10}-E_{10}^I)_x\Vert_{[x_{\frac{N}{4}-1},x_N]\times[0,1]}^2\leq C\Vert (E_{10})_x\Vert_{[x_{\frac{N}{4}-1},x_N]\times[0,1]}^2+C\Vert (E_{10}^I)_x\Vert_{[x_{\frac{N}{4}-1},x_N]\times[0,1]}^2\\
&\leq C\Vert (E_{10})_x\Vert_{[x_{\frac{N}{4}-1},x_N]\times[0,1]}^2+C\sum_{i=\frac{N}{4}-1}^{N-1}\sum_{j=0}^{N-1}\Vert (E_{10}^I)_x\Vert_{\mathscr{T}_{i,j}}^2\\
&\leq C\mu_0N^{-2\tau}+C\mu_1N^{2-2\tau}\\
&\leq C\mu_1N^{2-2\tau},
\end{aligned}
\end{equation}
where inverse inequality \cite[Theorem 3.2.6]{Ciarlet:2002-finite}, \eqref{eq:E10-property}, Lemmas \ref{eq:mesh-x-sizes} and \ref{eq:mesh-y-sizes} yield
\begin{equation*}
\begin{aligned}
&\sum_{i=\frac{N}{4}-1}^{N-1}\sum_{j=0}^{N-1}\Vert (E_{10}^I)_x\Vert_{\mathscr{T}_{i,j}}^2
\leq C\sum_{i=\frac{N}{4}-1}^{N-1}\sum_{j=0}^{N-1}h_{x,i}^{-2}\Vert E_{10}^I\Vert_{\mathscr{T}_{i,j}}^2\\
&\leq C\sum_{i=\frac{N}{4}-1}^{N-1}\sum_{j=0}^{N-1}h_{x,i}^{-2}\Vert E_{10}^I\Vert_{\infty,\mathscr{T}_{i,j}}^2h_{x,i}h_{y,j}\\
&\leq C\mu_1N^{2-2\tau}.
\end{aligned}
\end{equation*}

Similar to \eqref{eq:E10x-1}, we can get
\begin{equation}\label{eq:E10y-1}
\Vert (E_{10}-E_{10}^I)_y\Vert_{[x_0,x_{\frac{N}{4}-1}]\times[0,1]}^2\leq C\mu_0^{-1}N^{-2k}.
\end{equation}

Similar to \eqref{eq:E10x-2}, one has
\begin{equation}\label{eq:E10y-1}
\begin{aligned}
\Vert (E_{10}-E_{10}^I)_y\Vert_{[x_{\frac{N}{4}-1},x_N]\times[0,1]}^2
\leq C\varepsilon_1^{-\frac{1}{2}}N^{2-2\tau}.
\end{aligned}
\end{equation}
From \eqref{eq:E10x-1}--\eqref{eq:E10y-1}, \eqref{cos:conclusions-1}  and Lemma \ref{eq:Interpolation-L2} we can easily obtain 
\begin{equation*}
\begin{aligned}
\Vert E_{10}-E_{10}^I\Vert_E^2&\leq C(\varepsilon_1\mu_1N^{-2(k+1)}N^2+\varepsilon_1^{\frac{1}{2}}N^{-2(k+1)}N^2)+CN^{-2(k+1)}\\
&\leq C(\varepsilon_1\mu_1N^{-2k}+\varepsilon_1^{\frac{1}{2}}N^{-2k})+CN^{-2(k+1)}\\
&\leq C((\varepsilon_1^{\frac{1}{2}}+\varepsilon_2)N^{-2k}+\varepsilon_1^{\frac{1}{2}}N^{-2k})+CN^{-2(k+1)}\\
&\leq C(\varepsilon_1^{\frac{1}{2}}+\varepsilon_2)N^{-2k}+CN^{-2(k+1)},
\end{aligned}
\end{equation*}
i.e. 
\begin{equation*}\label{eq:E10--interpolation-energy-nome}
\Vert E_{10}-E_{10}^I\Vert_E\leq C(\varepsilon_1^{\frac{1}{2}}+\varepsilon_2)^{\frac{1}{2}}N^{-k}+CN^{-(k+1)}.
\end{equation*}

For $\Vert PE_{11}\Vert_E$ and $\Vert\Theta E_{11}\Vert_E$, Lemmas \ref{eq:mesh-x-sizes}, \ref{eq:mesh-y-sizes} , \eqref{cos:conclusions-1}  and \eqref{eq:varepsilon12-condition} yield
\begin{equation*}\label{eq:PE11-energy-nome}
\begin{aligned}
\Vert PE_{11}\Vert_E^2&\leq CN^{-2\tau}\sum_{s=1}^{k}\left(\sum_{j=0}^{N-1}\sum_{t=0}^{k-1}\Vert\theta_{\frac{3N}{4},j}^{s,t}\Vert_E^2+\Vert\theta_{\frac{3N}{4},N}^{s,0}\Vert_E^2\right)\\
&\leq CN^{-2\tau}\sum_{j=0}^{N-1}(\varepsilon_1h_{x,\frac{3N}{4}}^{-1}h_{y,j}+\varepsilon_1h_{x,\frac{3N}{4}}h_{y,j}^{-1}+h_{x,\frac{3N}{4}}h_{y,j})\\
&\leq CN^{-2\tau}(\varepsilon_1\mu_1+\varepsilon_1^{\frac{1}{2}}N+N^{-1}),\\
&\leq CN^{-2\tau}(\varepsilon_2+\varepsilon_1^{\frac{1}{2}}N+N^{-1})\\
&\leq CN^{-2\tau-1},
\end{aligned}
\end{equation*}
and
\begin{equation*}\label{eq:bound-E11-energy-nome}
\begin{aligned}
\Vert\Theta E_{11}\Vert_E^2&\leq CN^{-2\tau}\left(\sum_{s=1}^{k}\Vert\theta_{\frac{3N}{4},0}^{s,0}\Vert_E^2+\sum_{s=1}^{k}\Vert\theta_{\frac{3N}{4},N}^{s,0}\Vert_E^2\right)\\
&\leq CN^{-2\tau}(\varepsilon_1h_{x,\frac{3N}{4}}^{-1}h_{y,0}+\varepsilon_1h_{x,\frac{3N}{4}}h_{y,0}^{-1}+h_{x,\frac{3N}{4}}h_{y,0}\\
&+\varepsilon_1h_{x,\frac{3N}{4}}^{-1}h_{y,N-1}+\varepsilon_1h_{x,\frac{3N}{4}}h_{y,N-1}^{-1}+h_{x,\frac{3N}{4}}h_{y,N-1})\\
&\leq CN^{-2\tau}(\varepsilon_1\mu_1\varepsilon_1^{\frac{1}{2}}+\varepsilon_1N^{-1}\varepsilon_1^{-\frac{1}{2}}N+N^{-1}\varepsilon_1^{\frac{1}{2}})\\
&\leq CN^{-2\tau}(\varepsilon_1^{\frac{1}{2}}(\varepsilon_1^{\frac{1}{2}}+\varepsilon_2)+\varepsilon_1^{\frac{1}{2}})\\
&\leq C\varepsilon_1^{\frac{1}{2}}N^{-2\tau}.
\end{aligned}
\end{equation*}
\end{proof}
%

\begin{lemma}\label{eq:Interpolation-E-2}
For interpolation error estimates of corner layers we have
\begin{align*}
&\Vert E_i-E_i^I\Vert_E\leq CN^{-(k+1)} \quad i=31,32,33,34,\\
&\Vert PE_j\Vert_E\leq CN^{-\tau-\frac{1}{2}}\quad j=32,33,\\
&\Vert \Theta E_j\Vert_E\leq C\varepsilon_1^{\frac{1}{4}}N^{-\tau}\quad j=32,33.
\end{align*}
\end{lemma}
\begin{proof}
We have omitted the proofs of $\Vert PE_j\Vert_E$ and $\Vert \Theta E_j\Vert_E$ with $j=32,33$ here, because they are similar to ones of $\Vert PE_{11}\Vert_E$ and $\Vert \Theta E_{11}\Vert_E$, respectively.
In order to analyze  $ \Vert (E_{31}-E_{31}^I)\Vert_E$, we set $D_{0,0}:=[x_0,x_{\frac{N}{4}-1}]\times[y_0,y_{\frac{N}{4}-1}]$. Then
\begin{equation*}\label{eq:E31x-2}
\begin{aligned}
&\Vert (E_{31}-E_{31}^I)_x\Vert_{\Omega\backslash D_{0,0}}^2\leq C\Vert (E_{31})_x\Vert_{\Omega\backslash D_{0,0}}^2+C\Vert (E_{31}^I)_x\Vert_{\Omega\backslash D_{0,0}}^2\\
&\leq C\Vert (E_{31})_x\Vert_{\Omega\backslash D_{0,0}}^2
+C\sum_{i=\frac{N}{4}-1}^{N-1}\sum_{j=0}^{N-1}\Vert (E_{31}^I)_x\Vert_{\mathscr{T}_{i,j}}^2
+C\sum_{i=0}^{\frac{N}{4}-1}\sum_{j=\frac{N}{4}-1}^{N-1}\Vert (E_{31}^I)_x\Vert_{\mathscr{T}_{i,j}}^2\\
&\leq D_1+D_2+D_3.
\end{aligned}
\end{equation*}
Inverse inequality, \eqref{eq:E31-property}, Lemmas \ref{eq:mesh-x-sizes} and \ref{eq:mesh-y-sizes} yield
\begin{equation}\label{eq:D1}
\begin{aligned}
D_1&=\Vert (E_{31})_x\Vert_{\Omega\backslash D_{0,0}}^2\leq\int_{x_0}^{x_{\frac{N}{4}-1}}\int_{y_{\frac{N}{4}-1}}^{y_N}\mu_0^2e^{-2p\mu_0x_i}e^{-\frac{2\delta y_j}{\sqrt{\varepsilon_1}}}dxdy\\
&+\int_{x_{\frac{N}{4}-1}}^{x_N}\int_{y_0}^{y_N}\mu_0^2e^{-2p\mu_0x_i}e^{-\frac{2\delta y_j}{\sqrt{\varepsilon_1}}}dxdy\\
&\leq C\varepsilon_1^{\frac{1}{2}}\mu_0N^{-2\tau}.
\end{aligned}
\end{equation}
and
\begin{equation}\label{eq:D2}
\begin{aligned}
D_2&=\sum_{i=\frac{N}{4}-1}^{N-1}\sum_{j=0}^{N-1}\Vert (E_{31}^I)_x\Vert_{\mathscr{T}_{i,j}}^2
\leq C\sum_{i=\frac{N}{4}-1}^{N-1}\sum_{j=0}^{N-1}h_{x,i}^{-2}\Vert E_{31}^I\Vert_{\mathscr{T}_{i,j}}^2\\
&\leq C\sum_{i=\frac{N}{4}-1}^{N-1}\sum_{j=0}^{N-1}h_{x,i}^{-2}\Vert E_{31}^I\Vert_{\infty,\mathscr{T}_{i,j}}^2h_{x,i}h_{y,j}\\
&\leq C\mu_1N^{2-2\tau}.
\end{aligned}
\end{equation}
Similar to $D_2$, we have 
\begin{equation}\label{eq:D3}
\begin{aligned}
D_3
\leq C\mu_0N^{2-2\tau}.
\end{aligned}
\end{equation}
Combination of \eqref{eq:D1}, \eqref{eq:D2}, and \eqref{eq:D3} yields 
\begin{equation}\label{eq:E31x-2}
\Vert (E_{31}-E_{31}^I)_x\Vert_{\Omega\backslash D_{0,0}}^2\leq C\mu_1N^{2-2\tau}.
\end{equation}
From \eqref{eq:E31-property}, Lemma \ref{eq:Interpolation-error}, \eqref{lem:e-E10} with $m=l+\frac{1}{2}$ and \eqref{lem:e-E20} with $m=r+\frac{1}{2}$ yield
\begin{equation}\label{eq:E31x-1}
\begin{aligned}
&\Vert (E_{31}-E_{31}^I)_x\Vert_{D_{0,0}}^2
=\sum_{i=0}^{\frac{N}{4}-2}\sum_{j=0}^{\frac{N}{4}-2}\Vert(E_{31}-E_{31}^I)_x\Vert_{\mathscr{T}_{i,j}}^2\\
&\leq C\sum_{i=0}^{\frac{N}{4}-2}\sum_{j=0}^{\frac{N}{4}-2}
\sum\limits_{l+r=k}h_{x,i}^{2l}h_{y,j}^{2r}\left\Vert\frac{\partial^{k+1}E_{31}}{\partial x^{l+1}\partial y^r}\right\Vert_{\mathscr{T}_{i,j}}^2\\
&\leq C\sum_{i=0}^{\frac{N}{4}-2}\sum_{j=0}^{\frac{N}{4}-2}\sum\limits_{l+r=k}h_{x,i}^{2l}h_{y,j}^{2r}(\varepsilon_1^{-r}\mu_0^{2(l+1)}e^{-2p\mu_0x_i}e^{-\frac{2\delta y_j}{\sqrt{\varepsilon_1}}}h_{x,i}h_{y,j})\\
&\leq C\sum_{i=0}^{\frac{N}{4}-2}\sum_{j=0}^{\frac{N}{4}-2}\sum\limits_{l+r=k}(\mu_0^{-2l-1}N^{-2l-1}\varepsilon_1^{r+\frac{1}{2}}N^{-2r-1})\mu_0^{2(l+1)}\varepsilon_1^{-r}\\
&\leq C\sum_{i=0}^{\frac{N}{4}-2}\sum_{j=0}^{\frac{N}{4}-2}\mu_0\varepsilon_1^{\frac{1}{2}}N^{-2(k+1)}\\
&\leq C\mu_0\varepsilon_1^{\frac{1}{2}}N^{-2k}.
\end{aligned}
\end{equation}

For $\Vert (E_{31}-E_{31}^I)_y\Vert$, we use the same processing technique as $\Vert (E_{31}-E_{31}^I)_x\Vert$ to obtain
\begin{equation}\label{eq:E31y-1}
\begin{aligned}
\Vert (E_{31}-E_{31}^I)_y\Vert_{D_{0,0}}^2
\leq C\mu_0^{-1}\varepsilon_1^{-\frac{1}{2}}N^{-2k},
\end{aligned}
\end{equation}
\begin{equation}\label{eq:E31y-2}
\begin{aligned}
\Vert (E_{31}-E_{31}^I)_y\Vert_{\Omega\backslash D_{0,0}}^2
\leq C\varepsilon_1^{-\frac{1}{2}}N^{2-2\tau}.
\end{aligned}
\end{equation}
From \eqref{eq:E31x-2}--\eqref{eq:E31y-2} we can easily obtain 
\begin{equation*}\label{eq:H1-nome}
\vert E_{31}-E_{31}^I\vert_1^2\leq C \mu_1N^{-2k}+\varepsilon_1^{-\frac{1}{2}}N^{-2k}.
\end{equation*}
By combining Lemma \ref{eq:Interpolation-L2} and \eqref{cos:conclusions-1} we get 
\begin{equation*}\label{eq:E31-E31I-energy-nome}
\begin{aligned}
\Vert E_{31}-E_{31}^I\Vert_E^2&\leq C\varepsilon_1(\mu_1N^{-2k}+\varepsilon_1^{-\frac{1}{2}}N^{-2k})+CN^{-2(k+1)}\\
&\leq C(\varepsilon_1\mu_1N^{-2k}+\varepsilon_1^{\frac{1}{2}}N^{-2k})+CN^{-2(k+1)}\\
&\leq C((\varepsilon_1^{\frac{1}{2}}+\varepsilon_2)N^{-2k}+\varepsilon_1^{\frac{1}{2}}N^{-2k})+CN^{-2(k+1)}\\
&\leq C(\varepsilon_1^{\frac{1}{2}}+\varepsilon_2)N^{-2k}+CN^{-2(k+1)}.
\end{aligned}
\end{equation*}

Similarly, we have 
$$\Vert E_{i}-E_{i}^I\Vert_E^2\leq C(\varepsilon_1^{\frac{1}{2}}+\varepsilon_2)N^{-2k}+CN^{-2(k+1)}\qquad i=32,33,34.$$

\end{proof}

When calculating the interpolation error of $\Vert (E_i-E_i^I)_x\Vert(i=10,20,21,31,34)$, we use  a different technique. 
\begin{lemma}\label{interpolation-H-nome}
Assume $\tau\geq k+1$. On Bakhvalov-type mesh $\mathcal{T}$, one has
\begin{align*}
&\Vert (E_{10}-E_{10}^I)_x\Vert\leq C\mu_0N^{-(k+1)},\\
&\Vert (E_i-E_i^I)_x\Vert\leq C\varepsilon_1^{\frac{1}{4}}N^{-k}+CN^{-(k+1)}\quad i=20,21,\\
&\Vert (E_j-E_j^I)_x\Vert\leq C\mu_0^{\frac{1}{2}}\varepsilon_1^{\frac{1}{4}}N^{-k}+C\mu_0N^{-(k+1)}\quad j=31,34.
\end{align*}
\end{lemma}
\begin{proof}
Here we only prove the conclusion of the boundary layer at $x=0$,  because the proof for other boundary layers is similar. The analysis of the two corner layers is also similar, so we only present the proof of one of them.

For $\Vert (E_{10}-E_{10}^I)_x\Vert$, 
on the interval $[x_0,x_{\frac{N}{4}-1}]\times[0,1]$, using \eqref{eq:E10x-1} to get 
\begin{equation}\label{eq:E10x-1-b}
\begin{aligned}
\Vert (E_{10}-E_{10}^I)_x\Vert_{[x_0,x_{\frac{N}{4}-1}]\times[0,1]}^2
\leq C\mu_0N^{-2k}.
\end{aligned}
\end{equation}
But, on the interval $[x_{\frac{N}{4}-1},x_N]\times[0,1]$, instead of using the inverse inequality in \eqref{eq:E10x-2}, we use the triangle inequality and \eqref{eq:E10-property} yield 
\begin{equation}\label{eq:E10x-2-b}
\begin{aligned}
&\Vert (E_{10}-E_{10}^I)_x\Vert_{[x_{\frac{N}{4}-1},x_N]\times[0,1]}\\
&\leq \Vert (E_{10})_x\Vert_{\infty,[x_{\frac{N}{4}-1},x_N]\times[0,1]}+\Vert (E_{10}^I)_x\Vert_{\infty,[x_{\frac{N}{4}-1},x_N]\times[0,1]}\\
&\leq C\Vert (E_{10})_x\Vert_{\infty,[x_{\frac{N}{4}-1},x_N]\times[0,1]}\\
&\leq C\mu_0N^{-\tau}.
\end{aligned}
\end{equation}
From \eqref{eq:E10x-1-b} and \eqref{eq:E10x-2-b} we get $\Vert (E_{10}-E_{10}^I)_x\Vert\leq C\mu_0N^{-(k+1)}.$

For  $\Vert (E_{31}-E_{31}^I)_x\Vert$, we  decompose it as follows
\begin{equation*}\label{eq:d-E31}
\begin{aligned}
&\Vert (E_{31}-E_{31}^I)_x\Vert\leq \Vert (E_{31}-E_{31}^I)_x\Vert_{[x_0,x_{\frac{N}{4}-1}]\times[y_0,y_{\frac{N}{4}-1}]}\\
&+\Vert (E_{31}-E_{31}^I)_x\Vert_{[x_0,x_{\frac{N}{4}-1}]\times[y_{\frac{N}{4}-1},y_N]}+\Vert (E_{31}-E_{31}^I)_x\Vert_{[x_{\frac{N}{4}-1},x_N]\times[0,1]}\\
&\leq C\mu_0^{\frac{1}{2}}\varepsilon_1^{\frac{1}{4}}N^{-k}+C\mu_0N^{-(k+1)},
\end{aligned}
\end{equation*}
 where similar to \eqref{eq:E31x-1}, we have 
$$\Vert (E_{31}-E_{31}^I)_x\Vert_{[x_0,x_{\frac{N}{4}-1}]\times[y_0,y_{\frac{N}{4}-1}]}^2
\leq \mu_0\varepsilon_1^{\frac{1}{2}}N^{-2k},$$
and similar to \eqref{eq:E10x-2-b} we obtain
\begin{align*}
\Vert (E_{31}-E_{31}^I)_x\Vert_{[x_0,x_{\frac{N}{4}-1}]\times[y_{\frac{N}{4}-1},y_N]}+\Vert (E_{31}-E_{31}^I)_x\Vert_{[x_{\frac{N}{4}-1},x_N]\times[0,1]}\leq \mu_0N^{-\tau}.
\end{align*}

\end{proof}

\begin{theorem}\label{eq:interpolation-u-uI}
Assume $\tau\geq k+1$. On the Bakhvalov-type mesh $\mathcal{T}$, one has 
\begin{align*}
&\sum\limits_i\Vert \pi_iE_i-E_i\Vert\leq CN^{-(k+1)}, \qquad i=11,32,33,\\ 
&\Vert u-u^I\Vert_E+\Vert u-\Pi u\Vert_E\leq C(\varepsilon_1^{\frac{1}{2}}+\varepsilon_2)^{\frac{1}{2}}N^{-k}+N^{-(k+1)}.
\end{align*}
\end{theorem}
\begin{proof}
From  Lemma \ref{eq:Interpolation-L2} and the proof of Lemma \ref{eq:Interpolation-E-energy} we could obtain 
$$\Vert \pi E_{11}-E_{11}\Vert \leq \Vert E_{11}-E_{11}\Vert+\Vert PE_{11}\Vert++\Vert \Theta E_{11}\Vert \leq CN^{-(k+1)}.$$
Similarly we could get estimates for $\Vert \pi E_i-E_i\Vert$ with $i=32,33.$

By a simple calculation we get $\Vert S-S^I\Vert\leq CN^{-(k+1)}$ and $\vert S-S^I\vert_1\leq CN^{-k}$. Then by combining Lemmas \ref{eq:Interpolation-L2} ,\ref{eq:Interpolation-E-energy} and \ref{eq:Interpolation-E-2} we prove $\Vert u-u^I\Vert_E\leq C(\varepsilon_1^{\frac{1}{2}}+\varepsilon_2)^{\frac{1}{2}}N^{-k}+N^{-(k+1)}$.  Finally using \eqref{eq:piu}, we have
$\Vert u-\Pi u\Vert_E\leq C(\varepsilon_1^{\frac{1}{2}}+\varepsilon_2)^{\frac{1}{2}}N^{-k}+N^{-(k+1)}.$ 
\end{proof}

\section{Uniform convergence}

Set $\chi:=\Pi u-u^N$. Using \eqref{eq:bilinear form}, \eqref{eq:coercivity}, \eqref{eq:Interpolation-u}, integration by parts  and Galerkin's orthogonality  we have
\begin{equation*}\label{eq:uniform convergence-chi}
\begin{aligned}
&\alpha\Vert\chi\Vert_E^2\leq a(\chi,\chi)=a(\Pi u-u,\chi)\\
&=\varepsilon_1\int_{\Omega}\nabla(\Pi u-u)\nabla\chi dxdy+\varepsilon_2\int_{\Omega}(S^I-S)_x\chi dxdy\\
&+\sum\limits_{l=10,20,21,31,34}\varepsilon_2\int_{\Omega}(E_l^I-E_l)_x\chi dxdy-\varepsilon_2\int_{\Omega}(\pi_{11}E_{11}-E_{11})b\chi_xdxdy\\
&-\sum\limits_{i=32,33}\varepsilon_2\int_{\Omega}(\pi_iE_i-E_i)b\chi_xdxdy
-\sum\limits_{j=11,32,33}\varepsilon_2\int_{\Omega}(\pi_jE_j-E_j)b_x\chi dxdy+\int_{\Omega}c(\Pi u-u)\chi dx dy\\
&=:I+II+III+IV+V+VI+VII.
\end{aligned}
\end{equation*}

Theorem \ref{eq:interpolation-u-uI}  yields
\begin{equation}\label{eq:I-VI-VII}
\begin{aligned}
\vert (I+VII)+VI\vert&
\leq C\Vert \Pi u-u\Vert_E\Vert\chi\Vert_E+\sum\limits_{j=11,32,33}\Vert \pi_jE_j-E_j\Vert\Vert\chi\Vert\\
&\leq C((\varepsilon_1^{\frac{1}{2}}+\varepsilon_2)^{\frac{1}{2}}N^{-k}+N^{-(k+1)})\Vert\chi\Vert_E.
\end{aligned}
\end{equation}

Using \eqref{cos:conclusions-2}, Lemma \ref{interpolation-H-nome} and H\"{o}lder inequality we can get
\begin{equation}\label{eq:III}
\begin{aligned}
\vert II+III\vert&\leq C(\varepsilon_2\Vert(S^I-S)_x\Vert\Vert\chi\Vert+\varepsilon_2\sum\limits_{l=10,20,21,31,34}\Vert(E_l^I-E_l)_x\Vert\Vert\chi\Vert)\\
&\leq C\varepsilon_2(N^{-k}+\mu_0N^{-(k+1)}+\mu_0^{\frac{1}{2}}\varepsilon_1^{\frac{1}{4}}N^{-k})\Vert\chi\Vert\\
&\leq C(\varepsilon_2N^{-k}+N^{-(k+1)}+\varepsilon_2^{\frac{1}{2}}\varepsilon_1^{\frac{1}{4}}N^{-k})\Vert\chi\Vert\\
&\leq C(\varepsilon_2N^{-k}+N^{-(k+1)}+\varepsilon_2^{\frac{1}{2}}\varepsilon_1^{\frac{1}{4}}N^{-k})\Vert\chi\Vert_E.
\end{aligned}
\end{equation}

For $IV$ and $V$,  we have the following two lemmas.
\begin{lemma}\label{eq:convection item-IV}
Assuming that $\tau\ge k+1$, on the Bakhvalov-type mesh $\mathcal{T}$, one has 
\begin{equation*}
\vert IV\vert\leq C(\varepsilon_2N^{-k}+\varepsilon_2^{\frac{1}{2}}N^{-(k+1)})\Vert\chi\Vert_E.
\end{equation*}
\end{lemma}
\begin{proof}
After analysis, we do the following decomposition
\begin{equation}\label{eq:II11-decomposition}
\begin{aligned}
&\int_{\Omega}(\pi_{11}E_{11}-E_{11})b\chi_xdxdy=\int_{x_0}^{x_{\frac{3N}{4}}}\int_0^1 b(E_{11}^I-E_{11})\chi_x dxdy\\
&+\int_{x_{\frac{3N}{4}}}^{x_{\frac{3N}{4}+1}}\int_0^1 b(\pi_{11} E_{11}-E_{11})\chi_x dxdy
+\int_{x_{\frac{3N}{4}+1}}^{x_{\frac{3N}{4}+2}}\int_0^1 b(\pi_{11} E_{11}-E_{11})\chi_x dxdy\\
&+\int_{x_{\frac{3N}{4}+2}}^{x_N}\int_0^1 b(E_{11}^I-E_{11})\chi_x dxdy
=: F_1+ F_2+ F_3+ F_4.
\end{aligned}
\end{equation}

 First, using \eqref{eq:E11-property}, the inverse inequality,  Lemmas \ref{eq:mesh-x-sizes}, \ref{eq:mesh-y-sizes} and \ref{eq:Interpolation-error} we can obtain 
\begin{equation}\label{eq:O1}
\begin{aligned}
\vert F_1\vert&\leq \sum_{i=0}^{\frac{3N}{4}-1}\sum_{j=0}^{N-1}\Vert E_{11}^I-E_{11}\Vert_{\mathscr{T}_{i,j}}\Vert \chi_x\Vert
\leq C\sum_{i=0}^{\frac{3N}{4}-1}\sum_{j=0}^{N-1}\sum\limits_{l+r=k+1}h_{x,i}^lh_{y,j}^r\left\|\frac{\partial^{k+1}E_{11}}{\partial x^l\partial y^r}\right\|_{\mathscr{T}_{i,j}}\Vert \chi_x\Vert_{\mathscr{T}_{i,j}}\\
&\leq C\sum_{i=0}^{\frac{3N}{4}-1}\sum_{j=0}^{N-1}\sum\limits_{l+r=k+1}h_{x,i}^lh_{y,j}^r\mu_1^le^{-p\mu_1(1-x_{i+1})}h_{x,i}^{\frac{1}{2}}h_{y,j}^{\frac{1}{2}}h_{x_i}^{-1}\Vert \chi\Vert_{\mathscr{T}_{i,j}}\\
&\leq C\sum_{i=0}^{\frac{3N}{4}-1}\sum_{j=0}^{N-1}\sum\limits_{l+r=k+1}\mu_1^l\mu_1^{-\tau}h_{x,i}^{l-\frac{1}{2}}h_{y,j}^{r+\frac{1}{2}}\Vert \chi\Vert_{\mathscr{T}_{i,j}}
\leq C\mu_1^{l-\tau}\sum_{i=0}^{\frac{3N}{4}-1}\sum_{j=0}^{N-1}N^{-(k+1)}\Vert \chi\Vert_{\mathscr{T}_{i,j}}\\
&\leq C\mu_1^{l-\tau}\left(\sum_{i=0}^{\frac{3N}{4}-1}\sum_{j=0}^{N-1}N^{-2(k+1)}\right)^{\frac{1}{2}}\left(\sum_{i=0}^{\frac{3N}{4}-1}\sum_{j=0}^{N-1}\Vert \chi\Vert_{\mathscr{T}_{i,j}}^2\right)^{\frac{1}{2}}\\
&\leq CN^{-k}\Vert \chi\Vert_{[x_0,x_{\frac{3N}{4}}]\times[0,1]}
\leq CN^{-k}\Vert \chi\Vert_{E,[x_0,x_{\frac{3N}{4}}]\times[0,1]}.
\end{aligned}
\end{equation}

Next, using \eqref{eq:E11-property}, \eqref{lem:e-E11} with $m=l$,  Lemmas \ref{eq:mesh-x-sizes}, \ref{eq:mesh-y-sizes} and \ref{eq:Interpolation-error} we can obtain 
\begin{equation}\label{eq:O4}
\begin{aligned}
\vert F_4\vert&\leq \sum_{i=\frac{3N}{4}+2}^{N-1}\sum_{j=0}^{N-1}\Vert E_{11}^I-E_{11}\Vert_{\mathscr{T}_{i,j}}\Vert \chi_x\Vert_{\mathscr{T}_{i,j}}
\leq C\sum_{i=\frac{3N}{4}+2}^{N-1}\sum_{j=0}^{N-1}\sum\limits_{l+r=k+1}h_{x,i}^lh_{y,j}^r\left\|\frac{\partial^{k+1}E_{11}}{\partial x^l\partial y^r}\right\|_{\mathscr{T}_{i,j}}\Vert \chi_x\Vert_{\mathscr{T}_{i,j}}\\
&\leq C\sum_{i=\frac{3N}{4}+2}^{N-1}\sum_{j=0}^{N-1}\sum\limits_{l+r=k+1}h_{x,i}^lh_{y,j}^r\mu_1^le^{-p\mu_1(1-x_{i+1})}h_{x,i}^{\frac{1}{2}}h_{y,j}^{\frac{1}{2}}\Vert \chi_x\Vert_{\mathscr{T}_{i,j}}\\
&\leq C\sum_{i=\frac{3N}{4}+2}^{N-1}\sum_{j=0}^{N-1}\sum\limits_{l+r=k+1}(\mu_1^{-l}N^{-l})\mu_1^lh_{x,i}^{\frac{1}{2}}h_{y,j}^{r+\frac{1}{2}}\Vert \chi_x\Vert_{\mathscr{T}_{i,j}}
\leq C\mu_1^{-\frac{1}{2}}\sum_{i=\frac{3N}{4}+2}^{N-1}\sum_{j=0}^{N-1}N^{-(k+1)-\frac{1}{2}}\Vert \chi_x\Vert_{\mathscr{T}_{i,j}}\\
&\leq C\mu_1^{-\frac{1}{2}}\left(\sum_{i=\frac{3N}{4}+2}^{N-1}\sum_{j=0}^{N-1}N^{-2(k+1)-1}\right)^{\frac{1}{2}}\left(\sum_{i=\frac{3N}{4}+2}^{N-1}\sum_{j=0}^{N-1}\Vert \chi_x\Vert_{\mathscr{T}_{i,j}}^2\right)^{\frac{1}{2}}\\
&\leq C\mu_1^{-\frac{1}{2}}N^{-(k+\frac{1}{2})}\Vert \chi_x\Vert_{[x_{\frac{3N}{4}+2},x_N]\times[0,1]}\\
&\leq C\varepsilon_1^{-\frac{1}{2}}\mu_1^{-\frac{1}{2}}N^{-(k+\frac{1}{2})}\Vert \chi\Vert_{E,[x_{\frac{3N}{4}+2},x_N]\times[0,1]}.
\end{aligned}
\end{equation}

Then, on the interval $[x_{\frac{3N}{4}+1},x_{\frac{3N}{4}+2}]\times[0,1]$, we notice that\\ $$\pi_{11}E_{11}=E_{11}^I-\sum_{j=0}^{N-1}\sum_{t=0}^{k-1}E_{11}(x_{\frac{3N}{4}+1}^0,y_j^t)\theta_{\frac{3N}{4}+1,j}^{0,t}-E_{11}(x_{\frac{3N}{4}+1}^0,y_N^0)\theta_{\frac{3N}{4}+1,N}^{0,0}.$$
Thus we have 
\begin{equation}\label{eq:O2}
\begin{aligned}
&\vert F_3\vert
\leq C\sum_{j=0}^{N-1}\Vert E_{11}^I-E_{11}\Vert_{\mathscr{T}_{\frac{3N}{4}+1,j}}\Vert \chi_x\Vert_{\mathscr{T}_{\frac{3N}{4}+1,j}}\\
&+C\sum_{j=0}^{N-1}\sum_{t=0}^{k-1}\vert E_{11}(x_{\frac{3N}{4}+1}^0,y_j^t)\vert\Vert\theta_{\frac{3N}{4}+1,j}^{0,t}\Vert_{\mathscr{T}_{\frac{3N}{4}+1,j}}\Vert \chi_x\Vert_{\mathscr{T}_{\frac{3N}{4}+1,j}}\\
&+\vert E_{11}(x_{\frac{3N}{4}+1}^0,y_N^0)\vert\Vert\theta_{\frac{3N}{4}+1,N}^{0,0}\Vert_{\mathscr{T}_{\frac{3N}{4}+1,N-1}}\Vert \chi_x\Vert_{\mathscr{T}_{\frac{3N}{4}+1,N-1}}\\
&=:\mathcal{R}_1+\mathcal{R}_2+\mathcal{R}_3\\
&\leq  C\varepsilon_1^{-\frac{1}{2}}\mu_1^{-\frac{1}{2}}N^{-(k+1)}\Vert \chi\Vert_{E,[x_{\frac{3N}{4}+1},x_{\frac{3N}{4}+2}]\times[0,1]}
\end{aligned}
\end{equation}
where same as \eqref{eq:O4}, we get
\begin{equation*}\label{eq:R1}
\begin{aligned}
\mathcal{R}_1
\leq C\varepsilon_1^{-\frac{1}{2}}\mu_1^{-\frac{1}{2}}N^{-(k+1)}\Vert \chi\Vert_{E,[x_{\frac{3N}{4}+1},x_{\frac{3N}{4}+2}]\times[0,1]},
\end{aligned}
\end{equation*}
and \eqref{eq:E11-property}, Lemmas \ref{eq:mesh-x-sizes} and \ref{eq:mesh-y-sizes} yield
\begin{equation}\label{eq:R2}
\begin{aligned}
\mathcal{R}_2
&\leq C\sum_{j=0}^{N-1}N^{-\tau}h_{x,\frac{3N}{4}+1}^{\frac{1}{2}}h_{y,j}^{\frac{1}{2}}\Vert \chi_x\Vert_{\mathscr{T}_{\frac{3N}{4}+1,j}}\leq C \mu_1^{-\frac{1}{2}}\sum_{j=0}^{N-1}N^{-\tau}N^{-\frac{1}{2}}\Vert \chi_x\Vert_{\mathscr{T}_{\frac{3N}{4}+1,j}}\\
&\leq C\mu_1^{-\frac{1}{2}}\left(\sum_{j=0}^{N-1}N^{-2\tau-1}\right)^{\frac{1}{2}}\left(\sum_{j=0}^{N-1}\Vert \chi_x\Vert_{\mathscr{T}_{\frac{3N}{4}+1,j}}^2\right)^{\frac{1}{2}}\\
&\leq C\mu_1^{-\frac{1}{2}}N^{-\tau}\Vert \chi_x\Vert_{[x_{\frac{3N}{4}+1},x_{\frac{3N}{4}+2}]\times[0,1]}\\
&\leq C\varepsilon_1^{-\frac{1}{2}}\mu_1^{-\frac{1}{2}}N^{-\tau}\Vert \chi\Vert_{E,[x_{\frac{3N}{4}+1},x_{\frac{3N}{4}+2}]\times[0,1]}.
\end{aligned}
\end{equation}
In the same way, one has
\begin{equation*}\label{eq:R3}
\begin{aligned}
\mathcal{R}_3
\leq C\varepsilon_1^{-\frac{1}{2}}\mu_1^{-\frac{1}{2}}N^{-(\tau+\frac{1}{2})}\Vert \chi\Vert_{E,\mathscr{T}_{\frac{3N}{4}+1,N-1}}.
\end{aligned}
\end{equation*}

Last, for $ F_2$, on the interval $[x_{\frac{3N}{4}},x_{\frac{3N}{4}+1}]\times[0,1]$, 
\begin{equation*}
\begin{aligned}
\pi_{11}E_{11}=\sum_{j=0}^{N-1}\sum_{t=0}^{k-1}E_{11}(x_{\frac{3N}{4}}^0,y_j^t)\theta_{\frac{3N}{4},j}^{0,t}+\sum_{s=1}^{k}E_{11}(x_{\frac{3N}{4}}^s,y_0^0)\theta_{\frac{3N}{4},0}^{s,0}+\sum_{s=0}^{k}E_{11}(x_{\frac{3N}{4}}^s,y_N^0)\theta_{\frac{3N}{4},N}^{s,0}.
\end{aligned}
\end{equation*}
Thus we have 
\begin{equation}\label{eq:O3}
\begin{aligned}
&\vert F_2\vert
\leq C\sum_{j=0}^{N-1}\sum_{t=0}^{k-1}\vert E_{11}(x_{\frac{3N}{4}}^0,y_j^t)\vert\Vert\theta_{\frac{3N}{4},j}^{0,t}\Vert_{\mathscr{T}_{\frac{3N}{4},j}}\Vert \chi_x\Vert_{\mathscr{T}_{\frac{3N}{4},j}}\\
&+C\Vert E_{11}\Vert_{[x_{\frac{3N}{4}},x_{\frac{3N}{4}+1}]\times[0,1]}\Vert \chi_x\Vert_{[x_{\frac{3N}{4}},x_{\frac{3N}{4}+1}]\times[0,1]}\\
&+C\sum_{s=1}^{k}\vert E_{11}(x_{\frac{3N}{4}}^s,y_0^0)\vert\Vert\theta_{\frac{3N}{4},0}^{s,0}\Vert_{\mathscr{T}_{\frac{3N}{4},0}}\Vert\chi_x \Vert_{\mathscr{T}_{\frac{3N}{4},0}}\\
&+C\sum_{s=0}^{k}\vert E_{11}(x_{\frac{3N}{4}}^s,y_N^0)\vert\Vert\theta_{{\frac{3N}{4}},N}^{s,0}\Vert_{\mathscr{T}_{\frac{3N}{4},N-1}}\Vert\chi_x \Vert_{\mathscr{T}_{\frac{3N}{4},N-1}}\\
&=\mathcal{S}_1+\mathcal{S}_2+\mathcal{S}_3+\mathcal{S}_4\\
&\leq C(\varepsilon_1^{-\frac{1}{2}}\mu_1^{-\frac{1}{2}}N^{-(k+1)}+C\varepsilon_1^{-\frac{1}{4}}\mu_1^{-\frac{1}{4}}N^{-(\tau+\frac{1}{4})})\Vert \chi\Vert_{E,[x_{\frac{3N}{4}},x_{\frac{3N}{4}+1}]\times[0,1]}.
\end{aligned}
\end{equation}
The proof is as follows, same as \eqref{eq:R2}, we get
\begin{equation*}\label{eq:W1-w2}
\begin{aligned}
\mathcal{S}_1+\mathcal{S}_2\leq C\varepsilon_1^{-\frac{1}{2}}\mu_1^{-\tau}N^{-\frac{1}{2}}\Vert \chi\Vert_{E,[x_{\frac{3N}{4}},x_{\frac{3N}{4}+1}]\times[0,1]}.
\end{aligned}
\end{equation*}
When dealing with $\mathcal{S}_3$ and $\mathcal{S}_4$, the mesh scale in Lemma \ref{eq:mesh-x-sizes} is not enough, so we still use the analysis method of \eqref{eq:R2}, but we use the mesh scale of Lemma \ref{eq:better-mesh} with $\eta=\frac{1}{2}$, thus we have
\begin{equation*}\label{eq:W3}
\begin{aligned}
\mathcal{S}_3+\mathcal{S}_4\leq C\varepsilon_1^{-\frac{1}{4}}\mu_1^{-\frac{1}{4}}N^{-(\tau+\frac{1}{4})}\Vert \chi\Vert_{E,[x_{\frac{3N}{4}},x_{\frac{3N}{4}+1}]\times[0,1]}.
\end{aligned}
\end{equation*}

So, by combining \eqref{cos:conclusions-2}, \eqref{eq:mu01-assume},  \eqref{eq:II11-decomposition}, \eqref{eq:O1}, \eqref{eq:O4}, \eqref{eq:O2} and  \eqref{eq:O3} to obtain
\begin{equation*}\label{eq:IV}
\begin{aligned}
\vert IV\vert&\leq C(\varepsilon_2N^{-k}+\varepsilon_2\varepsilon_1^{-\frac{1}{2}}\mu_1^{-\frac{1}{2}}N^{-(k+\frac{1}{2})}+C\varepsilon_2\varepsilon_1^{-\frac{1}{4}}\mu_1^{-\frac{1}{4}}N^{-(k+\frac{5}{4})})\Vert \chi\Vert_{E}\\
&\leq C(\varepsilon_2N^{-k}+\varepsilon_2^{\frac{1}{2}}N^{-(k+\frac{1}{2})}+\varepsilon_2^{\frac{3}{4}}N^{-(k+\frac{5}{4})})\Vert \chi\Vert_{E}\\
&\leq C(\varepsilon_2N^{-k}+\varepsilon_2^{\frac{1}{2}}N^{-(k+\frac{1}{2})})\Vert \chi\Vert_{E}.
\end{aligned}
\end{equation*}
\end{proof}
\begin{lemma}\label{eq:convection item-V}
Assuming that $\tau\ge k+1$, on the Bakhvalov-type mesh $\mathcal{T}$, one has 
\begin{equation*}
\vert V\vert\leq C\varepsilon_2^{\frac{1}{2}}N^{-(k+\frac{1}{2})}\Vert\chi\Vert_E.
\end{equation*}
\end{lemma}
\begin{proof}
To simplify the analysis, we decompose $V$ as follows
\begin{equation*}\label{eq:d-E32}
\begin{aligned}
&V
=\varepsilon_2\int_{x_0}^{x_{\frac{3N}{4}}}\int_0^1(E_{32}^I-E_{32})b\chi_x dxdy
+\varepsilon_2\int_{x_{\frac{3N}{4}}}^{x_{\frac{3N}{4}+1}}\int_0^1(\pi_{32}E_{32}-E_{32})b\chi_x dxdy\\
&+\varepsilon_2\int_{x_{\frac{3N}{4}+1}}^{x_{\frac{3N}{4}+2}}\int_0^1(\pi_{32}E_{32}-E_{32})b\chi_x dxdy
+\varepsilon_2\int_{x_{\frac{3N}{4}+2}}^{x_N}\int_0^1(E_{32}^I-E_{32})b\chi_x dxdy\\
&=:M_1+M_2+M_3+M_4.
\end{aligned}
\end{equation*}

For $M_1$, using triangle inequality,   \eqref{cos:conclusions-2} and \eqref{eq:E32-property}   we can get
\begin{equation}\label{eq:M1}
\begin{aligned}
\vert M_1\vert&\leq C\varepsilon_2(\Vert E_{32}^I\Vert_{\infty,[x_0,x_{\frac{3N}{4}}]\times[0,1]}+\Vert E_{32}\Vert_{\infty,[x_0,x_{\frac{3N}{4}}]\times[0,1]})\Vert \chi_x\Vert_{[x_0,x_{\frac{3N}{4}}]\times[0,1]}\\
&\leq C\varepsilon_2\Vert E_{32}\Vert_{\infty,[x_0,x_{\frac{3N}{4}}]\times[0,1]}\Vert \chi_x\Vert_{E,[x_0,x_{\frac{3N}{4}}]\times[0,1]}\\
&\leq C\varepsilon_2\varepsilon_1^{-\frac{1}{2}}\mu_1^{-\tau}\Vert \chi\Vert_{E,[x_0,x_{\frac{3N}{4}}]\times[0,1]}\\
&\leq C\varepsilon_2^{\frac{1}{2}}\mu_1^{\frac{1}{2}-\tau}\Vert \chi\Vert_{E,[x_0,x_{\frac{3N}{4}}]\times[0,1]}.
\end{aligned}
\end{equation}

On the interval $[x_{\frac{3N}{4}},x_{\frac{3N}{4}+1}]\times[0,1]$, we notice 
$$\pi_{32}E_{32}=\sum_{j=0}^{N-1}\sum_{t=0}^{k-1}E_{32}(x_{\frac{3N}{4}}^0,y_j^t)\theta_{\frac{3N}{4},j}^{0,t}+\sum_{s=1}^{k}E_{32}(x_{\frac{3N}{4}}^s,y_0^0)\theta_{\frac{3N}{4},0}^{s,0}+\sum_{s=0}^{k}E_{32}(x_{\frac{3N}{4}}^s,y_N^0)\theta_{\frac{3N}{4},N}^{s,0}.$$
Thus 
\begin{equation}\label{eq:M2}
\begin{aligned}
\vert M_2\vert&\leq C\varepsilon_2\sum_{j=0}^{N-1}\sum_{t=0}^{k-1}\vert E_{32}(x_{\frac{3N}{4}}^0,y_j^t)\vert\Vert\theta_{\frac{3N}{4},j}^{0,t}\Vert_{\mathscr{T}_{\frac{3N}{4},j}}\Vert\chi_x\Vert_{\mathscr{T}_{\frac{3N}{4},j}}\\
&+C\varepsilon_2\sum_{s=1}^{k}\vert E_{32}(x_{\frac{3N}{4}}^s,y_0^0)\vert\Vert\theta_{\frac{3N}{4},0}^{s,0}\Vert_{\mathscr{T}_{\frac{3N}{4},0}}\Vert\chi_x\Vert_{\mathscr{T}_{\frac{3N}{4},0}}\\
&+C\varepsilon_2\sum_{s=0}^{k}\vert E_{32}(x_{\frac{3N}{4}}^s,y_N^0)\vert\Vert\theta_{\frac{3N}{4},N}^{s,0}\Vert_{\mathscr{T}_{\frac{3N}{4},N-1}}\Vert\chi_x\Vert_{\mathscr{T}_{\frac{3N}{4},N-1}}\\
&+C\varepsilon_2\Vert E_{32}\Vert_{[x_{\frac{3N}{4}},x_{\frac{3N}{4}+1}]\times[0,1]}\Vert\chi_x\Vert_{[x_{\frac{3N}{4}},x_{\frac{3N}{4}+1}]\times[0,1]}\\
&=:\mathcal{V}_1+\mathcal{V}_2+\mathcal{V}_3+\mathcal{V}_4\\
&\leq C(\varepsilon_2^{\frac{1}{2}}N^{-\tau-\frac{1}{2}}+\varepsilon_2^{\frac{1}{2}}\varepsilon_1^{\frac{1}{4}}N^{-\tau})\Vert\chi\Vert_{E},
\end{aligned}
\end{equation}
where similar to  \eqref{eq:R2}, one has $\mathcal{V}_1\leq C\varepsilon_2\varepsilon_1^{-\frac{1}{2}}\mu_1^{-\frac{1}{2}}\mu_1^{-\tau+\frac{1}{2}}N^{-1}\Vert\chi\Vert_E,$ then use \eqref{cos:conclusions-2} we obtain 
\begin{equation}\label{eq:V1}
\mathcal{V}_1\leq C\varepsilon_2^{\frac{1}{2}}\mu_1^{-\tau+\frac{1}{2}}N^{-1}\Vert\chi\Vert_E.
\end{equation}
 In the same way, one has 
\begin{align*}
&\mathcal{V}_2+\mathcal{V}_3\leq C\varepsilon_2N^{-\tau-\frac{1}{2}}\Vert\chi\Vert_{E}.\\
&\mathcal{V}_4\leq C\varepsilon_2^{\frac{1}{2}}\varepsilon_1^{\frac{1}{4}}N^{-\tau}\Vert\chi\Vert_E.
\end{align*}

On the interval $[x_{\frac{3N}{4}+1},x_{\frac{3N}{4}+2}]\times[0,1]$, 
$$\pi_{32}E_{32}=E_{32}^I-\sum_{j=0}^{N-1}\sum_{t=0}^{k-1}E_{32}(x_{\frac{3N}{4}+1}^0,y_j^t)\theta_{\frac{3N}{4}+1,j}^{0,t}-E_{32}(x_{\frac{3N}{4}+1}^0,y_N^0)\theta_{\frac{3N}{4}+1,N}^{0,0}.$$
Thus
\begin{equation}\label{eq:M3}
\begin{aligned}
\vert M_3\vert&\leq C\varepsilon_2\sum_{j=0}^{N-1}\Vert E_{32}^I-E_{32}\Vert_{\infty,\mathscr{T}_{\frac{3N}{4}+1,j}}h_{x,\frac{3N}{4}+1}^{\frac{1}{2}}\Vert\chi_x\Vert_{\mathscr{T}_{\frac{3N}{4}+1,j}}\\
&+C\varepsilon_2\sum_{j=0}^{N-1}\sum_{t=0}^{k-1}\vert E_{32}(x_{\frac{3N}{4}+1}^0,y_j^t)\vert\Vert\theta_{\frac{3N}{4}+1,j}^{0,t}\Vert_{\mathscr{T}_{\frac{3N}{4}+1,j}}\Vert\chi_x\Vert_{\mathscr{T}_{\frac{3N}{4}+1,j}}\\
&+C\varepsilon_2\vert E_{32}(x_{\frac{3N}{4}+1}^0,y_N^0)\vert\Vert\theta_{\frac{3N}{4}+1,N}^{0,0}\Vert_{\mathscr{T}_{\frac{3N}{4}+1,N-1}}\Vert\chi_x\Vert_{\mathscr{T}_{\frac{3N}{4}+1,N-1}}\\
&=:\mathcal{W}_1+\mathcal{W}_2+\mathcal{W}_3\leq C\varepsilon_2^{\frac{1}{2}}N^{-(k+\frac{1}{2})},
\end{aligned}
\end{equation}
where same as \eqref{eq:M1}, one has
\begin{equation*}\label{eq:n1}
\begin{aligned}
\mathcal{W}_1\leq C\varepsilon_2^{\frac{1}{2}}N^{\frac{1}{2}-\tau}\Vert\chi\Vert_E,
\end{aligned}
\end{equation*}
and in the same way as \eqref{eq:V1}, it can  obtain
\begin{align*}
&\mathcal{W}_2\leq C\varepsilon_2^{\frac{1}{2}}N^{-\tau}\Vert\chi\Vert_E,\\
&\mathcal{W}_3\leq C\varepsilon_1^{\frac{1+\tau}{2}}\varepsilon_1^{\frac{1}{4}}N^{-\tau}\Vert\chi\Vert_E.
\end{align*}

For $M_4$, H\"{o}lder inequality yields
\begin{equation}\label{eq:M4}
\begin{aligned}
\vert M_4\vert&\leq C\varepsilon_2\Vert E_{32}^I-E_{32}\Vert_{[x_{\frac{3N}{4}+2},x_N]\times[y_0,y_{\frac{N}{4}-1}]}\Vert \chi_x\Vert_{[x_{\frac{3N}{4}+2},x_N]\times[y_0,y_{\frac{N}{4}-1}]}\\
&+C\varepsilon_2\Vert E_{32}^I-E_{32}\Vert_{[x_{\frac{3N}{4}+2},x_N]\times[y_{\frac{N}{4}-1},y_N]}\Vert \chi_x\Vert_{[x_{\frac{3N}{4}+2},x_N]\times[y_{\frac{N}{4}-1},y_N]}\\
&=:\mathcal{Z}_1+\mathcal{Z}_2\leq C(\varepsilon_2\mu_1^{-\frac{1}{2}}N^{-k}+\varepsilon_2^{\frac{1}{2}}N^{\frac{1}{2}-\tau})\Vert\chi\Vert_E,
\end{aligned}
\end{equation}
where similar to \eqref{eq:O4}, one has 
\begin{equation*}\label{eq:f1}
\begin{aligned}
\mathcal{Z}_1\leq C\varepsilon_2\mu_1^{-\frac{1}{2}}N^{-k}\Vert \chi\Vert_E,
\end{aligned}
\end{equation*}
and H\"{o}lder inequality, \eqref{eq:E32-property}, Lemma \ref{eq:mesh-x-sizes} and \eqref{cos:conclusions-2} yield
\begin{equation*}\label{eq:f2}
\begin{aligned}
\mathcal{Z}_2&\leq C\varepsilon_2\Vert E_{32}^I-E_{32}\Vert_{\infty,[x_{\frac{3N}{4}+2},x_N]\times[y_{\frac{N}{4}-1},y_N]}(1-x_{\frac{3N}{4}+2})^{\frac{1}{2}}\Vert \chi_x\Vert_{[x_{\frac{3N}{4}+2},x_N]\times[y_{\frac{N}{4}-1},y_N]}\\
&\leq C\varepsilon_2\Vert E_{32}\Vert_{\infty,[x_{\frac{3N}{4}+2},x_N]\times[y_{\frac{N}{4}-1},y_N]}\mu_1^{-\frac{1}{2}}(\ln N)^{\frac{1}{2}}\Vert \chi_x\Vert\\
&\leq C\varepsilon_2N^{\tau}\mu_1^{-\frac{1}{2}}N^{\frac{1}{2}}\Vert \chi_x\Vert\\
&\leq C\varepsilon_2\varepsilon_1^{-\frac{1}{2}}\mu_1^{-\frac{1}{2}}N^{\tau}N^{\frac{1}{2}}\Vert \chi\Vert_E\\
&\leq C\varepsilon_2^{\frac{1}{2}}N^{\frac{1}{2}-\tau}\Vert\chi\Vert_E.
\end{aligned}
\end{equation*}

Thus,  from \eqref{eq:mu01-assume}, \eqref{eq:varepsilon12-condition}, \eqref{eq:M1}, \eqref{eq:M2}, \eqref{eq:M3} and \eqref{eq:M4}  we can obtain our conclusion.
\end{proof}

Now we present the main conclusions of this paper.
\begin{theorem}
Assuming $\tau\geq k+1$. On the Bakhvalov-type mesh $\mathcal{T}$, and based on Assumption \ref{eq:priori estimates of u}, we have
\begin{align*}\label{eq:uI-uN}
&\Vert u^I-u^N\Vert_E+\Vert \Pi u-u^N\Vert_E\leq C(\varepsilon_1^{\frac{1}{2}}+\varepsilon_2)^{\frac{1}{2}}N^{-k}+CN^{-(k+1)},\\
&\Vert  u-u^N\Vert_E\leq C(\varepsilon_1^{\frac{1}{2}}+\varepsilon_2)^{\frac{1}{2}}N^{-k}+CN^{-(k+1)}.
\end{align*} 
\end{theorem}
\begin{proof}
From  \eqref{eq:I-VI-VII}, \eqref{eq:III}, Lemmas \ref{eq:convection item-IV} and \ref{eq:convection item-V} we can prove 
$$\Vert \Pi u-u^N\Vert_E\leq C(\varepsilon_1^{\frac{1}{2}}+\varepsilon_2)^{\frac{1}{2}}N^{-k}+CN^{-(k+1)}.$$
Combination of \eqref{eq:piu}, Lemmas \ref{eq:Interpolation-E-energy} and  \ref{eq:Interpolation-E-2} yields
$$\Vert u^I-u^N\Vert_E\leq C(\varepsilon_1^{\frac{1}{2}}+\varepsilon_2)^{\frac{1}{2}}N^{-k}+CN^{-(k+1)}.$$

Finally, using Theorem \ref{eq:interpolation-u-uI} we prove that
\begin{equation*}
\Vert  u-u^N\Vert_E\leq C(\varepsilon_1^{\frac{1}{2}}+\varepsilon_2)^{\frac{1}{2}}N^{-k}+CN^{-(k+1)}.
\end{equation*}
\end{proof}
\section{Numerical experiments}
The purpose of this section is to verify that our main conclusions are correct. In order to do so, we study the performance of the method when applied to the test problem
\begin{align}
&-\varepsilon_1\Delta u+\varepsilon_2(2-x)u_x+u=f(x,y)\quad \text{in}\ \Omega,\label{eq:test-1}\\
&u\vert_{\partial\Omega}=0\nonumber,
\end{align}
where choice of the right-hand side satisfies
\begin{equation*}\label{eq:real solution}
u(x,y)=\frac{1}{4}\left(1-e^{-\mu_0x}\right)\left(1-e^{-\mu_1(1-x)}\right)\left(1-e^{\frac{y}{\sqrt{\varepsilon_1}}}\right)\left(1-e^{\frac{(1-y)}{\sqrt{\varepsilon_1}}}\right),
\end{equation*}
with 
$$\mu_0=\frac{-\varepsilon_2+\sqrt{\varepsilon_2^2+\varepsilon_1}}{\varepsilon_1},\ \  \mu_1=\frac{\varepsilon_2+\sqrt{\varepsilon_2^2+4\varepsilon_1}}{2\varepsilon_1}.$$
is the exact solution.


In our example, we take $k=1,2,3$,  $p=0.5$, $\delta=0.25$, $N=2^3,\cdots,2^9$. 
Besides, we should choose the perturbation parameter range $R(\varepsilon_1, \varepsilon_2)$ that meets the conditions \eqref{eq:mu01-assume}, \eqref{eq:transition-points} and the mesh is completely in the Bakhvalov-type. 
Thus, for problem \eqref{eq:test-1}, the value range of disturbance parameter $R(\varepsilon_1, \varepsilon_2)$ should be
$$R(\varepsilon_1, \varepsilon_2)=\{(\varepsilon_1,\varepsilon_2)|0<\varepsilon_1\leq 10^{-6},\ 0<\varepsilon_2\leq 10^{-3}\}.$$
To be more general, we take $\varepsilon_1=1,10^{-2},10^{-4},10^{-6},10^{-8},10^{-10}$,  $\varepsilon_2=1,10^{-4},10^{-8}.$

For any fixed value of $k$ and $\varepsilon_2$, energy norm error estimation will be calculated by
\begin{equation*}\label{eq:energy-test}
e^N=\Vert u-u^N\Vert_E,
\end{equation*}
where $u$ is the exact solution given by \eqref{eq:test-1} and $u^N$ represents its numerical approximation.
And its corresponding convergence rate is
\begin{equation*}\label{eq:convergence-test}
p^N=\frac{\ln e^N-\ln e^{2N}}{\ln 2}.
\end{equation*}

In Tables \ref{table:1}--\ref{table:3}, we give the energy error estimations and convergence orders of $k=1$ and $\varepsilon_2=1,10^{-4},10^{-8}$. At the same time, we present the energy estimations in the cases of $k=2, \varepsilon_2=1,10^{-4},10^{-8}$ and $k=3, \varepsilon_2= 1,10^{-4},10^{-8}$  in  the figure below.
As can be seen from the chart, our conclusion is verified.
\begin{table}[h]
\caption{$\Vert u-u^N\Vert_E$ in the case of $\varepsilon_2=1$ and $k=1$}
\footnotesize
\begin{tabular*}{\textwidth}{@{}@{\extracolsep{\fill}} c ccccccc @{}}
\cline{1-8}{}
    \multirow{2}{*}{ $\varepsilon_1$ }&\multicolumn{7}{c}{$N$ }\\ 
\cline{2-8}                 &8        &16         &32        &64       &128        &256       &512       \\
\cline{1-8}
\multirow{2}{*}{ $1$ }      &0.13E-2  &0.66E-3    &0.33E-3   &0.16E-3  &0.82E-4    &0.41E-4   &0.21E-4\\
                            &1.00     &1.00       &1.00      &1.00     &1.00       &1.00      &--   \\
\cline{2-8}
\multirow{2}{*}{ $10^{-2}$ }&0.39E-1  &0.20E-1    &0.99E-2   &0.49E-2  &0.25E-2    &0.12E-2   &0.62E-3\\
                            &0.98     &1.00       &1.00      &1.00     &1.00       &1.00      &--   \\
\cline{2-8}
\multirow{2}{*}{ $10^{-4}$ }&0.46E-1  &0.23E-1    &0.11E-1   &0.57E-2  &0.29E-2    &0.14E-2   &0.72E-3\\
                            &0.99     &1.00       &1.00      &1.00     &1.00       &1.00      &--   \\
\cline{2-8}
\multirow{2}{*}{ $10^{-6}$ }&0.46E-1  &0.23E-1    &0.12E-1   &0.58E-2  &0.29E-2    &0.14E-2   &0.72E-3\\
                            &0.99     &1.00       &1.00      &1.00     &1.00       &1.00      &--   \\    
\cline{2-8}
\multirow{2}{*}{ $10^{-8}$ }&0.46E-1  &0.23E-1    &0.12E-1   &0.58E-2  &0.29E-2    &0.14E-2   &0.72E-3\\
                            &0.99     &1.00       &1.00      &1.00     &1.00       &1.00      &--   \\ 
\cline{2-8}
\multirow{2}{*}{ $10^{-10}$}&0.46E-1  &0.23E-1    &0.12E-1   &0.58E-2  &0.29E-2    &0.14E-2   &0.72E-3\\
                            &0.99     &1.00       &1.00      &1.00     &1.00       &1.00      &--   \\ 
\cline{1-8}{}
\end{tabular*}
\label{table:1}
\end{table}
\begin{table}[h]
\caption{$\Vert u-u^N\Vert_E$ in the case of $\varepsilon_2=10^{-4}$ and $k=1$}
\footnotesize
\begin{tabular*}{\textwidth}{@{}@{\extracolsep{\fill}} c ccccccc @{}}
\cline{1-8}{}
    \multirow{2}{*}{ $\varepsilon_1$ }&\multicolumn{7}{c}{$N$ }\\ 
\cline{2-8}                 &8        &16         &32        &64       &128        &256       &512       \\
\cline{1-8}
\multirow{2}{*}{ $1$ }      &0.22E-2  &0.11E-2    &0.55E-3   &0.28E-3  &014E-3    &0.69E-4   &0.35E-4\\
                            &1.01     &1.00       &1.00      &1.00     &1.00       &1.00      &--   \\
\cline{2-8}
\multirow{2}{*}{ $10^{-2}$ }&0.33E-1  &0.17E-1    &0.84E-2   &0.42E-2  &0.21E-2    &0.11E-2   &0.53E-3\\
                            &1.00    &1.00       &1.00      &1.00     &1.00       &1.00      &--   \\
\cline{2-8}
\multirow{2}{*}{ $10^{-4}$ }&0.26E-1  &0.12E-1    &0.58E-2   &0.29E-2  &0.14E-2    &0.72E-3   &0.36E-3\\
                            &1.13     &1.03       &1.01      &1.00     &1.00       &1.00      &--   \\
\cline{2-8}
\multirow{2}{*}{ $10^{-6}$ }&0.83E-2  &0.38E-2    &0.19E-2   &0.93E-3  &0.46E-3    &0.23E-3   &0.12E-3\\
                            &1.14     &1.03       &1.01      &1.00     &1.00       &1.00      &--   \\    
\cline{2-8}
\multirow{2}{*}{ $10^{-8}$ }&0.28E-2  &0.12E-2    &0.60E-3   &0.30E-3  &0.15E-3    &0.74E-4   &0.37E-4\\
                            &1.18     &1.04       &1.01      &1.00     &1.00       &1.00      &--   \\ 
\cline{2-8}
\multirow{2}{*}{ $10^{-10}$}&0.17E-2  &0.70E-3    &0.33E-3   &0.16E-3  &0.81E-4    &0.41E-4   &0.20E-4\\
                            &1.32     &1.08       &1.02      &1.00     &1.00       &1.00      &--   \\ 
\cline{1-8}{}
\end{tabular*}
\label{table:2}
\end{table}
\begin{table}[h]
\caption{$\Vert u-u^N\Vert_E$ in the case of $\varepsilon_2=10^{-8}$ and $k=1$}
\footnotesize
\begin{tabular*}{\textwidth}{@{}@{\extracolsep{\fill}} c ccccccc @{}}
\cline{1-8}{}
    \multirow{2}{*}{ $\varepsilon_1$ }&\multicolumn{7}{c}{$N$ }\\ 
\cline{2-8}                 &8        &16         &32        &64       &128        &256       &512       \\
\cline{1-8}
\multirow{2}{*}{ $1$ }      &0.22E-2  &0.11E-2    &0.55E-3   &0.28E-3  &014E-3    &0.69E-4   &0.35E-4\\
                            &1.01     &1.00       &1.00      &1.00     &1.00       &1.00      &--   \\
\cline{2-8}
\multirow{2}{*}{ $10^{-2}$ }&0.33E-1  &0.17E-1    &0.84E-2   &0.42E-2  &0.21E-2    &0.11E-2   &0.53E-3\\
                            &1.00    &1.00       &1.00      &1.00     &1.00       &1.00      &--   \\
\cline{2-8}
\multirow{2}{*}{ $10^{-4}$ }&0.26E-1  &0.12E-1    &0.58E-2   &0.29E-2  &0.14E-2    &0.72E-3   &0.36E-3\\
                            &1.13     &1.03       &1.01      &1.00     &1.00       &1.00      &--   \\
\cline{2-8}
\multirow{2}{*}{ $10^{-6}$ }&0.83E-2  &0.38E-2    &0.19E-2   &0.93E-3  &0.47E-3    &0.23E-3   &0.12E-3\\
                            &1.13     &1.03       &1.01      &1.00     &1.00       &1.00      &--   \\    
\cline{2-8}
\multirow{2}{*}{ $10^{-8}$ }&0.26E-2  &0.12E-2    &0.59E-3   &0.29E-3  &0.15E-3    &0.74E-4   &0.37E-4\\
                            &1.13     &1.03       &1.01      &1.00     &1.00       &1.00      &--   \\ 
\cline{2-8}
\multirow{2}{*}{ $10^{-10}$}&0.84E-3  &0.38E-3    &0.19E-3   &0.93E-4  &0.47E-4    &0.23E-4   &0.12E-4\\
                            &1.13     &1.03       &1.01      &1.00     &1.00       &1.00      &--   \\
\cline{1-8}{}
\end{tabular*}
\label{table:3}
\end{table}

\begin{figure}[h]
\begin{minipage}[t]{0.45\linewidth}
\centering
\includegraphics[width=2.35in,height=2.0in]{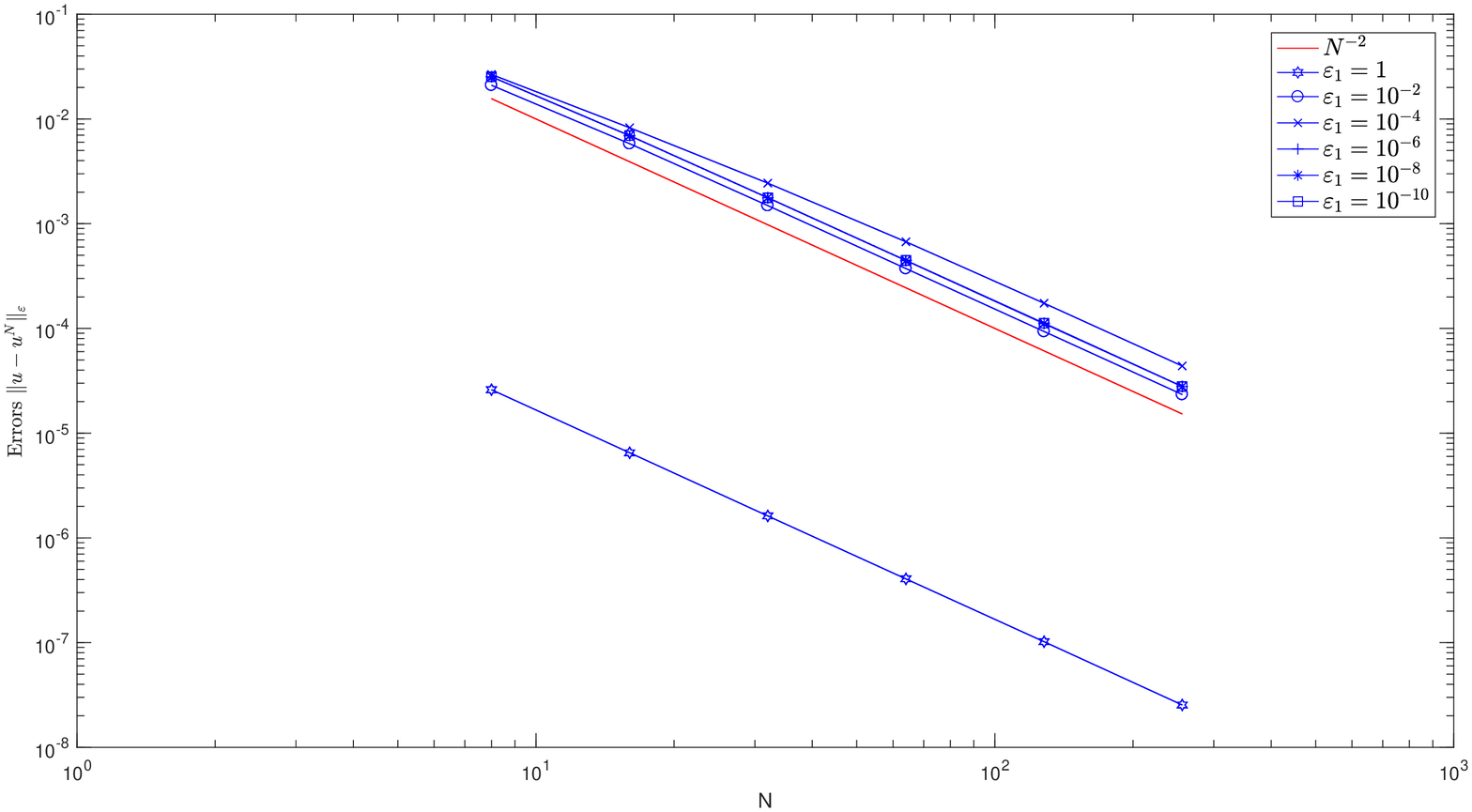}
\caption{ $k=2,\varepsilon_2=1$.}
\label{fig:1}
\end{minipage}
\begin{minipage}[t]{0.45\linewidth}
\centering
\includegraphics[width=2.35in,height=2.0in]{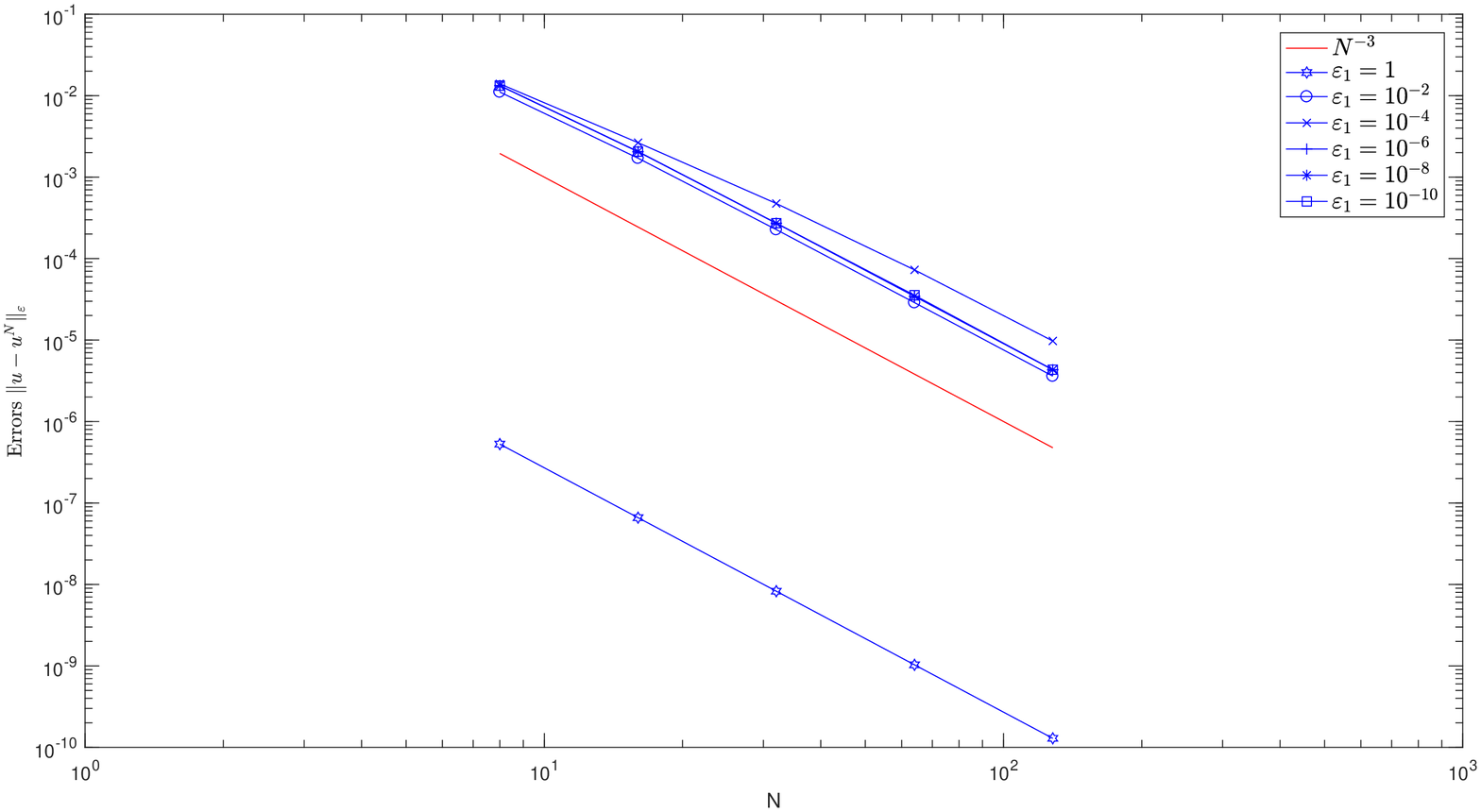}
\caption{  $k=3,\varepsilon_2=1$.}
\label{fig:2}
\end{minipage}
\end{figure}

\begin{figure}[h]
\begin{minipage}[t]{0.45\linewidth}
\centering
\includegraphics[width=2.35in,height=2.0in]{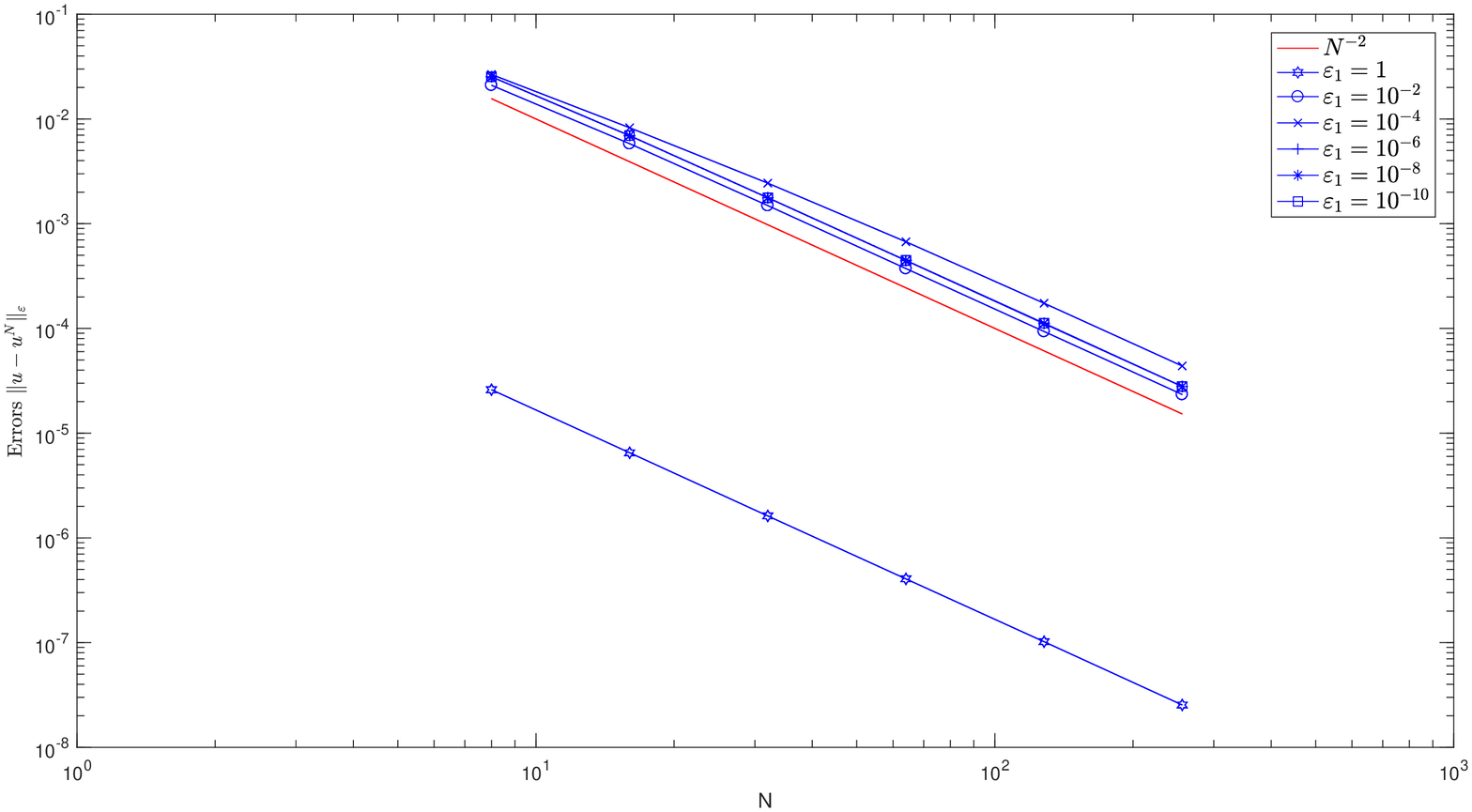}
\caption{ $k=2,\varepsilon_2=10^{-4}$.}
\label{fig:3}
\end{minipage}
\begin{minipage}[t]{0.45\linewidth}
\centering
\includegraphics[width=2.35in,height=2.0in]{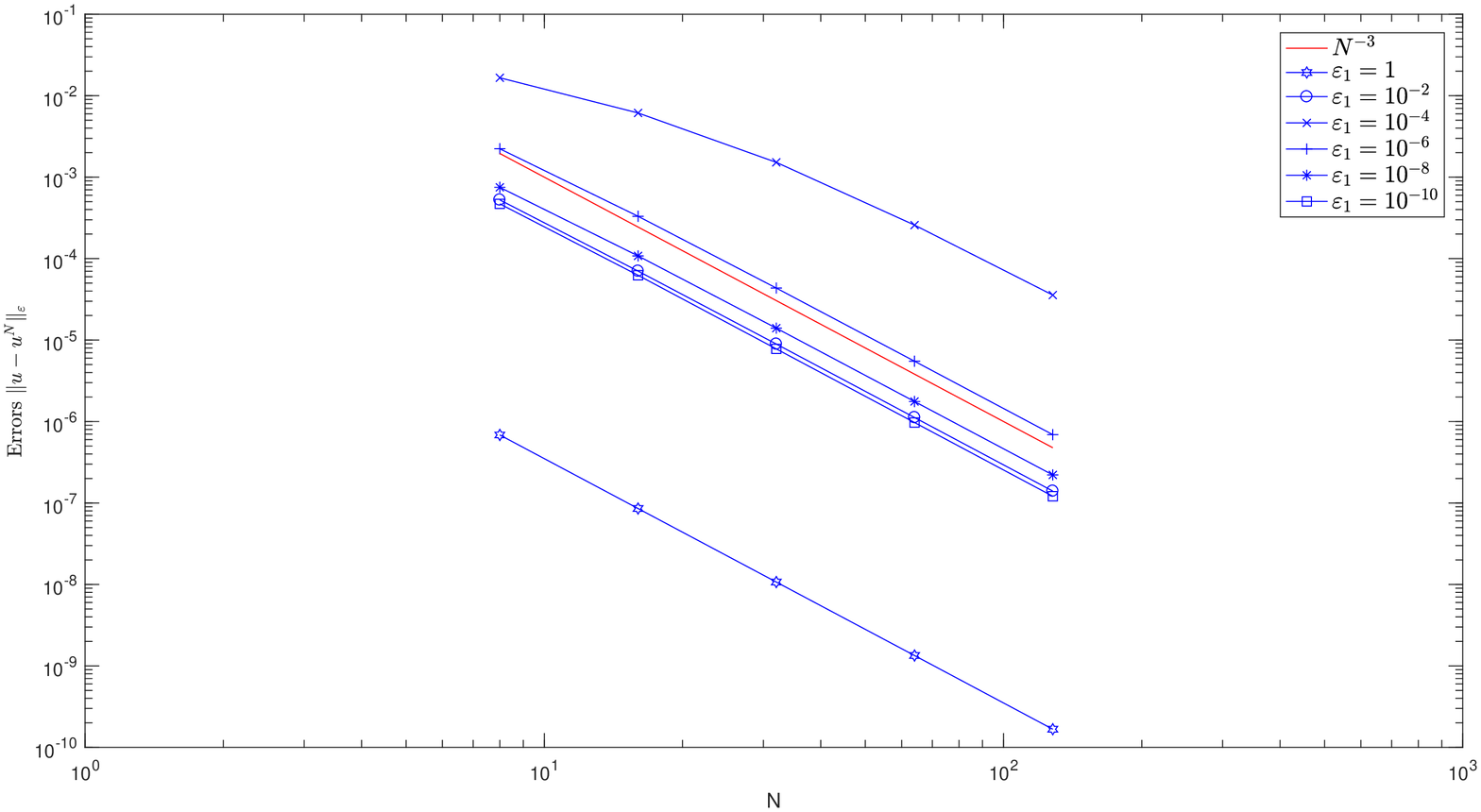}
\caption{  $k=3,\varepsilon_2=10^{-4}$.}
\label{fig:4}
\end{minipage}
\end{figure}

\begin{figure}[h]
\begin{minipage}[t]{0.45\linewidth}
\centering
\includegraphics[width=2.35in,height=2.0in]{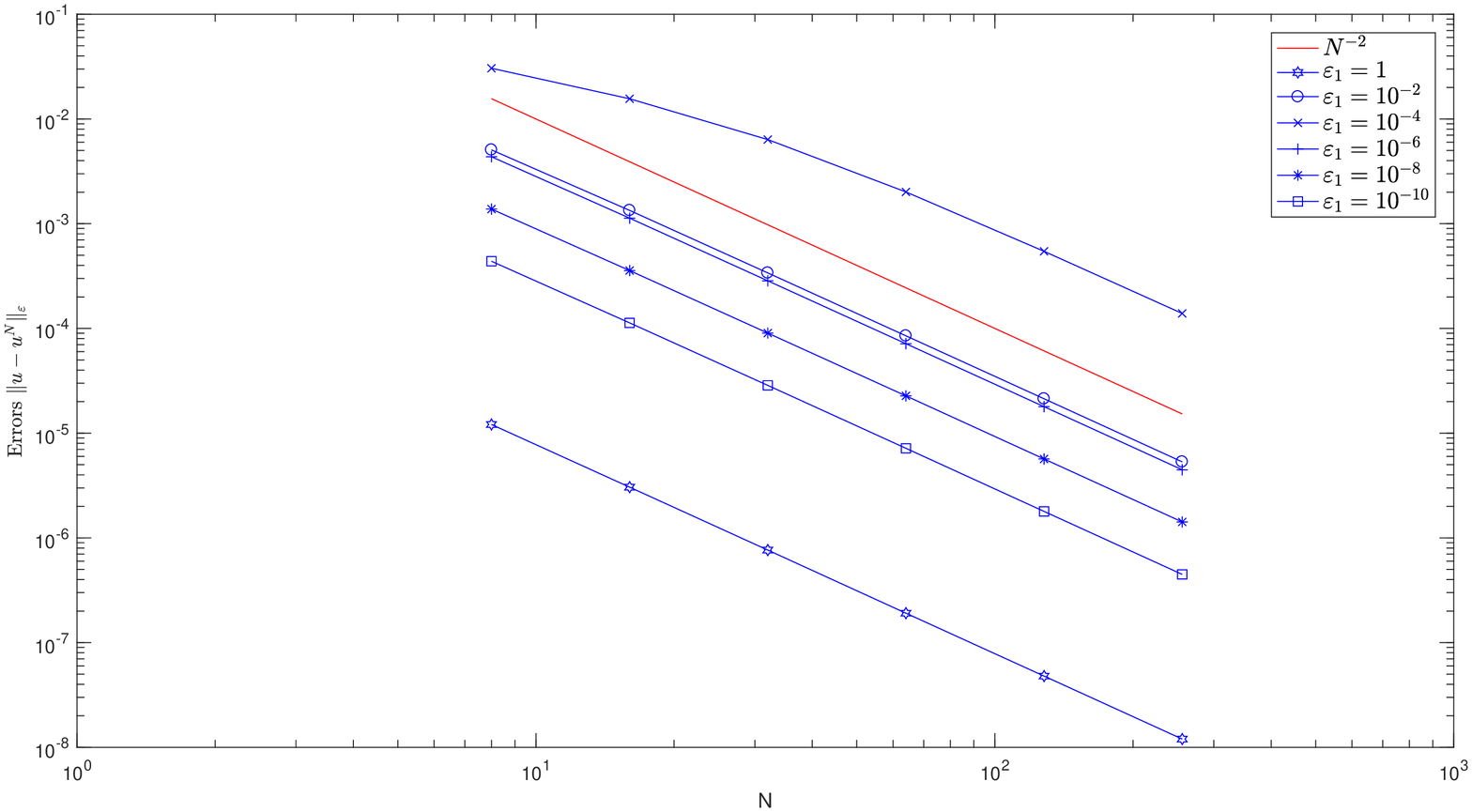}
\caption{ $k=2,\varepsilon_2=10^{-8}$.}
\label{fig:5}
\end{minipage}
\begin{minipage}[t]{0.45\linewidth}
\centering
\includegraphics[width=2.35in,height=2.0in]{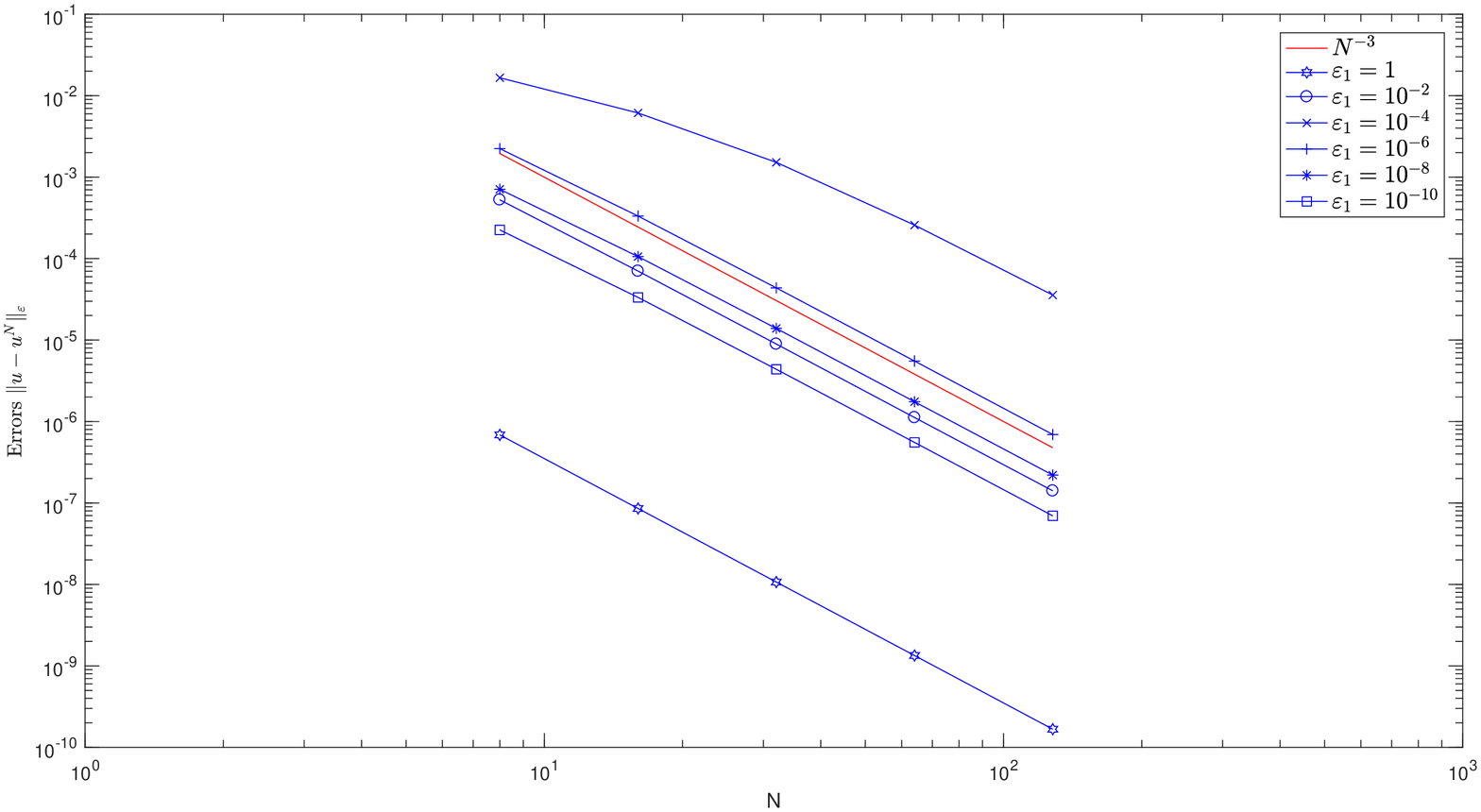}
\caption{$k=3,\varepsilon_2=10^{-8}$.}
\label{fig:6}
\end{minipage}
\end{figure}


\end{document}